\documentclass[11pt]{amsart}
\usepackage[utf8]{inputenc}
\usepackage{amsmath,amssymb}
\usepackage{enumerate}

\newtheorem{theorem}{Theorem}[section]

\newtheorem{lemma}[theorem]{Lemma}

\newtheorem{proposition}[theorem]{Proposition}

\newtheorem{corollary}[theorem]{Corollary}

{\theoremstyle{definition}}
{\theoremstyle{definition}}

{\theoremstyle{definition}\newtheorem{example}[theorem]{Example}}

{\theoremstyle{definition}\newtheorem{definition}[theorem]{Definition}}

{\theoremstyle{definition}}

{\theoremstyle{definition}\newtheorem{remark}[theorem]{Remark}}

\newcommand{\CC}{\mathbb{C}}
\newcommand{\NN}{\mathbb{N}}
\newcommand{\DD}{\mathbb{D}}
\newcommand{\RR}{\mathbb{R}}
\newcommand{\ZZ}{\mathbb{Z}}
\newcommand{\TT}{\mathbb{T}}

\newcommand{\veps}{\varepsilon}
\newcommand{\minv}{m_{\mathrm{inv}}}
\newcommand{\lfm}{\mathrm{LFM}(\mathbb D)}
\newcommand{\hyp}{\mathrm{Hyp}(\mathbb D)}
\newcommand{\para}{\mathrm{Par}(\mathbb D)}
\newcommand{\aut}{\mathrm{Aut}(\mathbb D)}
\newcommand{\ldens}{\underline{\textrm{dens}}}
\newcommand{\udens}{\overline{\textrm{dens}}}
\newcommand{\Ald}{\mathcal A_{ld}}
\newcommand{\Aud}{\mathcal A_{ud}}
\newcommand{\Aubd}{\mathcal A_{u\mathcal Bd}}

\begin{document}
\title[Disjoint frequent hypercyclicity]
{Disjoint frequent hypercyclicity of composition operators}
\date{\today}

\author[F. Bayart]{Frédéric Bayart}

\address{Laboratoire de Math\'ematiques Blaise Pascal UMR 6620 CNRS, Universit\'e Clermont Auvergne, Campus universitaire des C\'ezeaux, 3 place Vasarely, 63178 Aubi\`ere Cedex, France.}
\email{frederic.bayart@uca.fr}


\subjclass[2010]{}


\begin{abstract}
We give a sufficient condition for two operators to be disjointly frequently hypercyclic. We apply this criterion to composition operators acting on $H(\mathbb D)$ or on the Hardy space $H^2(\mathbb D)$. We simplify a result on disjoint frequent hypercyclicity of pseudo shifts of a recent paper of Martin et al. and we exhibit two disjointly frequently hypercyclic weighted shifts.
\end{abstract}
\maketitle

\section{Introduction}

Among the many notions in linear dynamics, frequent hypercyclicity (introduced in \cite{BAYGRITAMS}) and disjoint hypercyclicity
(introduced independently in \cite{Bernal-disjoint} and in \cite{BePe-disjoint}) play a central role. The first one quantifies how often
the orbit of a vector can visit each nonempty open set whereas the second one studies in some sense the independence of the orbits
of two hypercyclic operators.
Very recently, in  \cite{MMP22} and \cite{MaPu21}, these two classes were merged into the following natural definition.
Throughout this paper, the letter $X$ will mean a separable and infinite-dimensional $F$-space.
\begin{definition}
 Let $T_1,\dots,T_N\in\mathcal L(X)$. These operators are called {\bf disjointly frequently hypercyclic} if there is a vector $x\in X$
 such that, for any nonempty open set $U\subset X^N$, $\{n\in\NN:\ (T_1^n x,\dots,T_N^n x)\in U\}$ has positive lower density.
\end{definition}
We can similarly define disjoint upper frequent hypercyclicity and disjoint reiterative hypercyclicity by replacing the lower
density by the upper density or by the upper Banach density.

In \cite{MMP22} a criterion for the disjoint frequent (resp. upper frequent, reiterative) hypercyclicity of $N$ operators
is given. This criterion requires to construct a priori subsets of $\NN$ with positive lower (resp. upper, resp. upper Banach) 
density and which are well separated. Then the authors of \cite{MMP22}
deduce criteria for the disjoint frequent (resp. upper frequent, reiterative) hypercyclicity of pseudo-shifts $T_{f_1,w_1},\dots,T_{f_N,w_N}$.
Recall that, for $w=(w_j)$ a weight sequence and $f:\mathbb N\to\mathbb N$ increasing with $f(1)>1$,
the unilateral pseudo-shift $T_{f,w}$ is defined on $c_0(\NN)$ or $\ell^p(\NN)$ by
$$T_{f,w}\left(\sum_{j=1}^{\infty}x_je_j\right)=\sum_{j=1}^{\infty}w_{f(j)}x_{f(j)}e_j.$$
In particular, the class of pseudo-shifts encompasses that of weighted shifts for which $f(j)=j+1$. These criteria are rather technical
(in particular, one needs separation properties on $f_1,\dots,f_N$). However they allow the authors of \cite{MMP22}
to exhibit couples of weighted shifts $(B_{w_1},B_{w_2})$ which are disjointly upper frequently hypercyclic, but not disjointly frequently
hypercyclic, or which are disjointly reiteratively frequently hypercyclic and not disjointly upper frequently hypercyclic. 

\smallskip

Our first aim, in this paper, is to shed some new light on the results of \cite{MMP22}. We begin by deducing from Theorem 2.1 in \cite{MMP22}
another criterion for disjoint frequent hypercyclicity without requiring to exhibit disjoint sets with positive lower density.
Instead of this, we will need an increasing sequence $(n_k)$ with $n_k\leq Ck$ for some $C>0$. This criterion will be applied later on in this paper, with $(n_k)$ 
which will be sometimes the whole sequence of integers and sometimes a subsequence of $\mathbb N$.

We then study disjoint hypercyclicity of weighted shifts and pseudo-shifts of $\ell^p(\mathbb N)$. We exhibit an example of two weighted shifts
on $\ell^p(\mathbb N)$ which are disjointly frequently hypercyclic (such an example is not provided in \cite{MMP22}). 
We also show how to delete an extra technical assumption in the criteria of \cite{MMP22} for disjoint frequent hypercyclicity
of pseudo-shifts. For this we develop a technique which has further applications. Indeed, as pointed above, to construct (disjointly) frequently hypercyclic
vectors, we often need to exhibit subsets of $\NN$ with positive lower density and which are well-separated. We get here a very general 
result in that direction that we will use several times.

\smallskip

The remaining part of the paper is devoted to a thorough study of disjoint frequent hypercyclicity of another very important class of operators
in linear dynamics: composition operators. Let $\varphi$ be a self-map of the unit disc $\DD$ (in what follows, self-map will always mean holomorphic self-map).
The composition operator $C_\varphi(f)=f\circ\varphi$ defines a bounded operator both on $H^2(\DD)$ and on $H(\DD)$.
We first focus on $H^2(\DD)$. The study of the dynamical properties of $C_\varphi$ is
rather complicated and is only well understood
if $\varphi$ is a linear fractional map. We will denote by $\lfm$ the class of these maps.  In that case, we may always assume that $\varphi$ has no fixed point in $\DD$ 
so that it has an attractive fixed point on the unit circle $\TT$. We know that $C_\varphi$ is frequently hypercyclic if and only if $\varphi$ is 
hyperbolic or a parabolic automorphism (see \cite{BM09} for instance). Moreover, when $\varphi_1\neq\varphi_2$, $C_{\varphi_1}$ and $C_{\varphi_2}$ are
disjointly hypercyclic if and only if either they do not have the same attractive fixed point or they have the same attractive fixed point $\alpha\in\TT$ and $\varphi_1'(\alpha)\neq\varphi_2'(\alpha)$.
We complete the picture with a complete characterization of disjoint frequent hypercyclicity.

\begin{theorem}\label{thm:coh2}
Let $\varphi_1,\varphi_2\in \lfm$, with $\varphi_1\neq\varphi_2$ and denote by $\alpha_1,\alpha_2$ their respective attractive fixed points. Then $C_{\varphi_1}$ and 
$C_{\varphi_2}$ are disjointly frequently hypercyclic on $H^2(\DD)$ if and only if
\begin{enumerate}[(a)]
 \item each symbol $\varphi_i$ is either a parabolic automorphism or a hyperbolic linear fractional map;
 \item $\alpha_1\neq\alpha_2$ or $\alpha_1=\alpha_2=\alpha$ and
 \begin{enumerate}[(i)]
 \item either $\varphi_1$ and $\varphi_2$ are hyperbolic with $\varphi_1'(\alpha)\neq\varphi_2'(\alpha)$;
 \item or $\varphi_1$ and $\varphi_2$ are parabolic;
 \item or $\varphi_1$ is parabolic and $\varphi_2$ is hyperbolic and is not an automorphism;
 \item or $\varphi_2$ is parabolic and $\varphi_1$ is hyperbolic and is not an automorphism.
 \end{enumerate}
\end{enumerate}
\end{theorem}
Therefore, if $(\varphi_1,\varphi_2)\in \lfm\times \lfm$ is such that $C_{\varphi_1}$ and $C_{\varphi_2}$ are frequently hypercyclic, then 
\begin{itemize}
 \item either $\varphi_1$ and $\varphi_2$ are both hyperbolic with the same attractive fixed point and the same derivative at this point: in that case, $C_{\varphi_1}$
 and $C_{\varphi_2}$ are not even disjointly hypercyclic;
 \item or $\varphi_1$ and $\varphi_2$ share the same attractive fixed point, with $\varphi_1$ parabolic and $\varphi_2$ a hyperbolic automorphism, 
 or $\varphi_2$ parabolic and $\varphi_1$ a hyperbolic automorphism: in that case, $C_{\varphi_1}$ and $C_{\varphi_2}$ are disjointly hypercyclic, but not
 disjointly frequently hypercyclic;
 \item or $C_{\varphi_1}$ and $C_{\varphi_2}$ are disjointly frequently hypercyclic. 
\end{itemize}
The proof of Theorem \ref{thm:coh2} will be rather long. It will divided into several subcases, depending on the nature of the maps $\varphi_i$ and 
the position of their fixed points. Each subcase will require a different argument.
We will need several lemmas which are of independent interest. Indeed, we will have to quantify precisely how the iterates of a linear fractional
map converge to its attractive fixed point. We will also exhibit dense subsets of functions in $H^2(\DD)$ with prescribed set of zeros and with a rather precised
control of their behaviour near these zeros.

For our second result on disjointly frequently hypercyclic composition operators, we enlarge the space $H^2(\DD)$ to $H(\DD)$. In this space, it is easier
to be frequently hypercyclic (for instance, any parabolic linear fractional map of $\DD$ induces a frequently hypercyclic composition operator on $H(\DD)$,
even if it is not an automorphism). Disjoint hypercyclicity in this context was first studied in \cite{BeMaPe11} when the symbols
are linear fractional maps of $\DD$. We shall investigate disjoint frequent hypercyclicity. Using the linear fractional model of \cite{BoSh97},
we intend to state a result which can be applied to a much larger class of symbols (thus leading to new results even in the context of disjoint
hypercyclicity). For the relevant definitions, see Section \ref{sec:cohol}.

\begin{theorem}\label{thm:cohol}
 Let $\varphi_1$ and $\varphi_2$ be univalent and regular self-maps of $\DD$ with respective Denjoy-Wolff points $\alpha_1,\ \alpha_2\in\TT$.
 When $\alpha_1\neq\alpha_2$, $C_{\varphi_1}$ and $C_{\varphi_2}$ are always disjointly frequently hypercyclic. When $\alpha=\alpha_1=\alpha_2$,
 $C_{\varphi_1}$ and $C_{\varphi_2}$ are disjointly frequently hypercyclic if and only if
 \begin{enumerate}[(a)]
  \item either $\varphi_1$ and $\varphi_2$ are hyperbolic with $\varphi_1'(\alpha)\neq\varphi_2'(\alpha)$;
  \item or $\varphi_1$ and $\varphi_2$ are parabolic with $(\varphi_1''(\alpha),\varphi_1^{(3)}(\alpha))\neq (\varphi_2''(\alpha),\varphi_2^{(3)}(\alpha))$;
  \item or $\varphi_1$ is hyperbolic, $\varphi_2$ is parabolic and $\varphi_2''(\alpha)\in i\mathbb R$;
  \item or $\varphi_1$ is parabolic, $\varphi_2$ is hyperbolic and $\varphi_1''(\alpha)\in i\mathbb R$.
 \end{enumerate}

\end{theorem}

\noindent {\bf Notation.} Throughout this paper we use the notation $f\lesssim g$ for nonnegative
functions $f$ and $g$ to mean that there exists $C > 0$ such that $f\leq C g$, where $C$
does not depend on the associated variables.


\section{A criterion for disjoint frequent hypercyclicity}
We first state a criterion for disjoint frequent hypercyclicity.  Observe that it can be thought as a mixing of the disjoint 
hypercyclicity criterion and the frequent hypercyclicity criterion. We recall that, given a family $(x_{n,k})_{n\in \mathbb N,k\in I}\subset X$, the series $\sum_n x_{n,k}$ are said to be unconditionally convergent in $X$, uniformly in $k\in I$ if, for any $\veps>0$, there exists $N\geq 1$ such that, for any $k\in I$, any $F\subset [N,+\infty)$ finite, 
\begin{equation}\label{eq:unconditional}
 \left\|\sum_{n\in F}x_{n,k}\right\|\leq\veps.
\end{equation}
This implies in particular that given any $A\subset\mathbb N$, each series $\sum_{n\in A}x_{n,k}$ is convergent and \eqref{eq:unconditional}
remains true when $F\subset[N,+\infty)$ has infinite cardinal number.

\begin{theorem}\label{thm:dfhcc}
Let $X$ be a separable $F$-space, let $T_1,\dots,T_N\in\mathcal L(X)$. Assume that there exists a dense set $Y\subset X^N$, an increasing sequence of integers $(n_k)$ with positive lower density, maps $S_{n_k}:Y\to X$ such that, for all $y\in Y$, the following assumptions are satisfied:
\begin{itemize}
\item[(C1)] $\sum_{k=0}^{+\infty}S_{n_k}y$ converges unconditionally in $X$.
\item[(C2)] for all $j=1,\dots,N$, $\sum_{l=0}^{+\infty}T_j^{n_k}S_{n_{k+l}}y$ converges unconditionally in $X$, uniformly in $k\in\mathbb N$.
\item[(C3)] for all $j=1,\dots,N$, $\sum_{l=0}^{k}T_j^{n_k}S_{n_{k-l}}y$ converges unconditionally in $X$, uniformly in $k\in\mathbb N$.
\item[(C4)] for all $j=1,\dots,N$, $T_j^{n_k}S_{n_k}y\to y_j$.
\end{itemize}
Then $T_1,\dots,T_N$ are disjointly frequently hypercyclic.
\end{theorem}

\begin{proof}
Let $(\veps_p)$ be any sequence of positive real numbers which decreases to zero. Let $(y(p))_{p\in\mathbb N}$ be a dense sequence in $Y$. 
The assumptions give us, for each $p\geq 1$, the existence of an integer $N_p$ such that, for any $F\subset [N_p,+\infty)$, 
\begin{enumerate}[(i)]
\item for all $q\leq p$, $\left\|\sum_{i\in F}S_{n_i}y(q)\right\|<\veps_p$.
\item for all $j=1,\dots,N$, for all $k\geq 0$, for all $q\leq p$, 
$$\Bigg\|\sum_{l\in F}T_j^{n_k}S_{n_{k+l}}y(q)\Bigg\|<\veps_p.$$
\item for all $j=1,\dots,N$, for all $k\geq 0$, for all $q\leq p$, 
$$\Bigg\|\sum_{l\in F\cap [0,k]}T_j^{n_k}S_{n_{k-l}}y(q)\Bigg\|<\veps_p.$$
\item for all $j=1,\dots,N$, for all $k\geq N_p$, $\big\|T_j^{n_k}S_{n_k}y(p)-y_j(p)\big\|<\veps_p$.
\end{enumerate}
We then consider a sequence $(B_p)$ of disjoint subsets of $\mathbb N$ with positive lower density and such that $\min(B_p)\geq N_p$ and $|n-n'|\geq N_p+N_{p'}$ for all $n\in B_p$ and all $n'\in B_p$ with $n\neq n'$, for all $p,p'\geq 1$ (see \cite[Lemma 6.19]{BM09}). Let finally $A_p=\{n_k:\ k\in B_p\}$ which has positive lower density. We verify that the assumptions of Theorem 2.1 of \cite{MMP22} are satisfied. First of all, let $\veps>0$ and let $p_0$ be such that $\veps_{p_0}<\veps$. Then, for $p\leq p_0$, (i) yields for all $F\subset[N_{p_0},+\infty)$, 
$$\Bigg\|\sum_{n\in F\cap A_p} S_n y(p)\Bigg\|<\veps_{p_0}<\veps.$$
For $p>p_0$, since $\min(A_p)\geq N_p$, we also get by (i)
$$\Bigg\|\sum_{n\in F\cap A_p} S_n y(p)\Bigg\|<\veps_{p}<\veps.$$
Therefore the series $\sum_{n\in A_p}S_n y(p)$ converge unconditionally in $X$, uniformly in $p$.
Consider now $p\geq 1$ and let $j\in\{1,\dots,N\}$. Let $n\in\bigcup_{q\geq 1}A_q$ and write it $n=n_k$ for some $k$.  Then 
\begin{align*}
\Bigg\|\sum_{i\in A_p\backslash\{n\}} T_j^n S_i y(p)\Bigg\|&\leq 
\Bigg\|\sum_{\substack{i\in A_p, i>n}} T_j^n S_i y(p)\Bigg\|+\Bigg\|\sum_{\substack{i\in A_p, i<n}} T_j^n S_i y(p)\Bigg\|\\
&\leq \Bigg\|\sum_{l\in F_1}T_j^{n_k}S_{n_{k+l}}y(p)\Bigg\|+\Bigg\|\sum_{l\in F_2}T_j^{n_k}S_{n_{k-l}}y(p)\Bigg\|
\end{align*}
where both sets $F_1$ and $F_2$ are contained in $[N_p,+\infty)$ by the definition of $A_p$. Therefore, by (ii) and (iii),
$$\Bigg\|\sum_{i\in A_p\backslash\{n\}} T_j^n S_i y(p)\Bigg\|\leq 2\veps_p.$$
On the other hand, for  any $n\in A_p$, any $q<p$, we write 
$$\Bigg\|\sum_{i\in A_q}T_j^n S_i y(q)\Bigg\|\leq 
\Bigg\| \sum_{\substack{i\in A_q, i>n}}T_j^n S_i y(q)\Bigg\|+
\Bigg\| \sum_{\substack{i\in A_q, i<n}}T_j^n S_i y(q)\Bigg\|$$
and we find again that this is less than $2\veps_p$, applying again (ii) and (iii). Finally, the last assumption
of Theorem 2.1 of \cite{MMP22} follows directly from (iv), using $\min(A_p)\geq N_p$.
\end{proof}

To prove the uniform unconditional convergence of the series which come into play in the statement of Theorem \ref{thm:dfhcc}, we will use the following easy fact whose proof is left to the reader:
\begin{lemma}\label{lem:uniformunconditional}
Let $X$ be an $F$-space and $(y_l(k))_{l,k\geq 0}$ be vectors in $X$. 
\begin{enumerate}[(a)]
\item The series $\sum_{l=0}^{+\infty}y_l(k)$ converge unconditionally, uniformly in $k$, provided one of the following conditions is satisfied:
\begin{itemize}
\item there exists $(\veps_n)\in\ell^1(\mathbb N)$ such that,  for all $k, l\geq 0$, $\|y_l(k)\|\leq \veps_l$;
\item there exists $(\veps_n)\in c_0(\mathbb N)$ such that, for all $k\geq 0$, $\sum_l \|y_l(k)\|\leq \veps_k$;
\item there exists $(\veps_n)\in c_0(\mathbb N)$ such that, for all  $k\geq 0$,
for all $F\subset\NN$ finite, $\|\sum_{l\in F}y_l(k)\|\leq \veps_{\min(F)}$;
\item for all $k\geq 0$, $\sum_l y_l(k)$ converges unconditionally and there exists $(\veps_n)\in c_0(\mathbb N)$ such that, for all $k\geq 0$, for all $F\subset \mathbb N$ finite,
$\left\|\sum_{l\in F}y_l(k)\right\|\leq \veps_k$.
\end{itemize}
\item The series $\sum_{l=0}^k y_l(k)$ converge unconditionally, uniformly in $k$, provided one of the following conditions is satisfied:
\begin{itemize}
\item there exists $(\veps_n)\in\ell^1(\mathbb N)$ such that,  for all $k\geq l\geq 0$, $\|y_l(k)\|\leq \veps_l$;
\item  there exists $(\veps_n)\in c_0(\mathbb N)$ such that, for all $k\geq 0$,
for all $F\subset\NN$ finite, $\|\sum_{l\in F\cap [0,k]}y_l(k)\|\leq\veps_k$.
\end{itemize}
\end{enumerate}
\end{lemma}


\section{Disjoint frequent hypercyclicity of (pseudo)-shifts}

\subsection{An example of two disjoint frequently hypercyclic weighted shifts}

\begin{theorem}
There exist two disjointly frequently hypercyclic weighted shifts on $\ell^1(\mathbb N)$.
\end{theorem}

\begin{proof}
We will simultaneously construct two weighted shifts $B_w$ and $B_{w'}$ with respective weight sequence $w=(w_n)_{n\geq 1}$, $w'=(w'_n)_{n\geq 1}$
and a disjointly frequently hypercyclic vector $z=(z_n)_{n\geq 0}\in \ell^1(\mathbb N)$. Let $(N_p)$ be an increasing sequence of integers satisfying $N_p\geq p$.
Several additional assumptions on $(N_p)$ will be required throughout the proof. Let $(A_p)_{p\geq 1}$ be a sequence of disjoint subsets of $\mathbb N$ such that, for all $p\in\NN$, 
we have $\ldens(A_p)>0$, $\min(A_p)\geq p$ and if $n\in A_p$, $m\in A_q$ with $n\neq m$, then $|n-m|\geq N_p+N_q$ (see \cite[Lemma 6.19]{BM09}).
Enumerate $\bigcup_p A_p$ as an increasing sequence $(n_k)_{k\geq 1}$ and define $n_0=1$. Let finally $(x_p,y_p)$ be a dense sequence in $\ell^1(\mathbb N)\times\ell^1(\mathbb N)$ which may be written 
$$x_p=(a_0^{(p)},\dots,a_{p-1}^{(p)},0,\dots)\textrm{ and }y_p=(b_0^{(p)},\dots,b_{p-1}^{(p)},0,\dots)$$
with $1/p\leq |a_i^{(p)}|,|b_i^{(p)}|\leq p$ for all $0\leq i\leq p-1$. 
We will make our construction inductively. We initialize it by setting 
$$w_n=w'_n=4\textrm{ and }z_n=0\textrm{ for all }n<n_1.$$
Now, let $k\geq 1$ and assume that we have built our sequences up to $n_k-1$ such that, for all $j<k$, they satisfy
\begin{enumerate}
\item[(WS1)] $w_{n_j}\cdots w_{n_{j+1}-1}=w'_{n_j}\cdots w'_{n_{j+1}-1}=4^{n_{j+1}-n_j}$.
\item[(WS2)] if $n_j\in A_p$ then for all $l\in\{0,\dots,p-1\}$ we have
\begin{enumerate}[(a)]
\item $w_{1+l}\cdots w_{n_j+l}z_{n_j+l}=a_l^{(p)}$.
\item $w'_{1+l}\cdots w'_{n_j+l}z_{n_j+l}=b_l^{(p)}$.
\item $|z_{n_j+l}|\leq p^{1+2l}/4^{n_j-1}$.
\item $z_n=0$ if $n\in\{n_j+p,\cdots,n_{j+1}-1\}$.
\end{enumerate}
\item[(WS3)] for all $n\leq n_k-1$, $1/n\leq |w_n|,|w'_n|\leq 5$.
\item[(WS4)] for all $n\leq n_k-1$, $|w_1\cdots w_n|,|w'_1\cdots w'_n|\geq 3^n$. 
\end{enumerate}
Next, if $n_k\in A_p$, then set
$$z_{n_k}=\frac p{4^{n_k-1}},\ w_{n_k}=\frac{a_0^{(p)}}{p},\ w'_{n_k}=\frac{b_0^{(p)}}{p}$$
so that (WS2) (a), (b), (c) are satisfied for $l=0$ and $j=k$. Moreover $1/p^2\leq |w_{n_k}|,|w'_{n_k}|\leq 1$ so that (WS3) is satisfied with $n=n_k$ provided 
\begin{equation}\label{eq:dfhcnp1}
N_p\geq p^2.
\end{equation}
For $l=1,\dots,p-1$, we then define
$$z_{n_k+l}=\frac{p^{1+2l}}{4^{n_k-1}},\ w_{n_k+l}=\frac{a_l^{(p)}w_l}{a_{l-1}^{(p)}p^2},\ w'_{n_k+l}=\frac{b_l^{(p)}w_l}{b_{l-1}^{(p)}p^2}$$
so that again (WS2) (a), (b), (c) and (WS3) are satisfied, provided now 
\begin{equation}\label{eq:dfhcnp2}
N_p\geq p^5
\end{equation}
since $w_{n_k+l}\geq w_l/p^4\geq 1/p^5$. Lastly, for $n_k+p\leq n\leq n_{k+1}-1$, we set $z_n=0$ and 
$$w_n=\left(\frac{4^{n_{k+1}-n_k}}{w_{n_k}\cdots w_{n_k+p-1}}\right)^{\frac1{n_{k+1}-n_k-p}}\textrm{ and }w'_n=\left(\frac{4^{n_{k+1}-n_k}}{w'_{n_k}\cdots w'_{n_k+p-1}}\right)^{\frac1{n_{k+1}-n_k-p}}$$
so that $w_{n_k}\cdots w_{n_{k+1}-1}=w'_{n_k}\cdots w'_{n_{k+1}-1}=4^{n_{k+1}-n_k}$. Moreover since 
$$\frac{1}{(p-1)! p^{2+2p}}\leq w_{n_k}\cdots w_{n_k+p-1}=\frac{w_1\cdots w_{p-1} a_{p-1}^{(p)}}{p^{1+2p}}\leq 5^p$$
we get on the one hand, for $n\geq n_k+p$, 
\begin{align*}
w_n&\leq 4^{\frac{n_{k+1}-n_k}{n_{k+1}-n_k-p}}\left((p-1)!p^{2+2p}\right)^{\frac 1{n_{k+1}-n_k-p}}\\
&\leq 4^{\frac{N_p}{N_p-p}}\left((p-1)!p^{2+2p}\right)^{\frac 1{N_p-p}}\leq 5
\end{align*}
provided 
\begin{equation}\label{eq:dfhcnp3}
4^{\frac{N_p}{N_p-p}}\left((p-1)!p^{2+2p}\right)^{\frac 1{N_p-p}}\leq 5.
\end{equation}
On the other hand 
$$w_n\geq \left(\frac 15\right)^{\frac{p}{n_{k+1}-n_k-p}}\geq \left(\frac 15\right)^{\frac{p}{N_p-p}}\geq\frac 1p\geq \frac 1n$$
provided 
\begin{equation}\label{eq:dfhcnp4}
\left(\frac 15\right)^{\frac{p}{N_p-p}}\geq \frac 1p.
\end{equation}
To finish the construction, it remains to verify (WS4) for $n\in\{n_k,\cdots, n_{k+1}-1\}$. Firstly, if $n=n_k+l$, $l=0,\dots,p-1$, then 
\begin{align*}
w_1\cdots w_{n_k+l}&\geq \frac{4^{n_k}}{(p-1)!p^{2+2p}}\\
&\geq 3^{n_k+l}\left(\frac 43\right)^{n_k}\left(\frac 13\right)^{l}\frac{1}{(p-1)!p^{2+2p}}\\
&\geq 3^{n_k+l}\left(\frac 43\right)^{N_p}\left(\frac 13\right)^{p}\frac{1}{(p-1)!p^{2+2p}}\\
&\geq 3^{n_k+l}
\end{align*}
provided 
\begin{equation}\label{eq:dfhcnp5}
\left(\frac 43\right)^{N_p}\left(\frac 13\right)^{p}\frac{1}{(p-1)!p^{2+2p}}\geq 1.
\end{equation}
For $n\in\{n_k+p,\cdots,n_{k+1}-1\}$, we distinguish two cases. If $w_{n_k}\cdots w_{n_k+p-1}\leq 4^p$, then $w_m\geq 4$ for all $m\in\{n_k+p,\cdots,n_{k+1}-1\}$ so that
$$w_1\cdots w_n\geq 3^{n_k+p-1}4^{n-(n_k+p-1)}\geq 3^n.$$
If $w_{n_k}\cdots w_{n_k+p-1}\geq 4^p$, then we start from the stronger inequality 
$w_1\cdots w_{n_k+p-1}\geq 4^{n_k+p-1}$ and the definition of $w_n$ on $\{n_k+p,\cdots,n_{k+1}-1\}$ ensures that $w_1\cdots w_n\geq 4^n$ for all $n=n_k+p,\dots,n_{k+1}-1$.

\medskip

We now claim that $z\in\ell^1(\mathbb N)$ and that it is a disjointly frequently hypercyclic vector for $B_w$ and $B_{w'}$. First of all, observe that 
\begin{align*}
\|z\|_1&=\sum_{p\geq 1}\sum_{n\in A_p}\sum_{l=0}^{p-1} |z_{n+l}|\\
&\leq \sum_{p\geq 1}\sum_{n\in A_p}\frac{p^{2+2p}}{4^{n-1}}\\
&\leq \frac 43\sum_{p\geq 1}\frac{p^{2+2p}}{4^{N_p-1}}<+\infty
\end{align*}
if we require
\begin{equation}\label{eq:dfhcnp6}
\sum_{p\geq 1}\frac{p^{2+2p}}{4^{N_p-1}}<+\infty.
\end{equation}
Finally, let $n\in A_p$. Then 
\begin{align*}
\|B_w^n z-x(p)\|_1&=\sum_{q\geq 1} \sum_{\substack{m\in A_q\\ m>n}} \sum_{l=0}^{q-1}|w_{m+l-n+1}\cdots w_{m+l}z_{m+l}|\\
&=\sum_{q\geq 1} \sum_{\substack{m\in A_q\\ m>n}} \sum_{l=0}^{q-1}\frac{|w_1\cdots w_l|\cdot |w_{1+l}\cdots w_{m+l}z_{m+l}|}{|w_1\cdots w_{m+l-n}|}\\
&\leq \sum_{q\geq 1} \sum_{\substack{m\in A_q\\ m>n}} \frac{5^qq^2}{3^{m-n}}.
\end{align*}
We split the sum into two parts. 
\begin{align*}
\sum_{q\leq p}\sum_{\substack{m\in A_q\\ m>n}} \frac{5^qq^2}{3^{m-n}}&\leq
\sum_{q\leq p}\sum_{s\geq N_p}\frac{5^pp^2}{3^s}\leq \frac 32\cdot \frac{5^p p^3}{3^{N_p}}
\end{align*}
whereas 
\begin{align*}
\sum_{q>p} \sum_{\substack{m\in A_q\\ m>n}} \frac{5^qq^2}{3^{m-n}}
&\leq \sum_{q>p}\sum_{s\geq N_q} \frac{5^qq^2}{3^s}\\
&\leq \frac 32\sum_{q>p} \frac{5^q q^2}{3^{N_q}}.
\end{align*}
Therefore if we require moreover 
\begin{equation}\label{eq:dfhcnp7}
\frac{5^pp^3}{3^{N_p}}\to_{p\to+\infty}0\textrm{ and }\sum_{q>p}\frac{5^q q^2}{3^{N_q}}\to_{p\to+\infty}0,
\end{equation}
we get $\|B_w^n z-x(p)\|\leq\veps(p)$, $\|B_{w'}^n z-y(p)\|\leq\veps(p)$ with $\veps(p)\to 0$ as $p\to+\infty$ which concludes the proof. Observe that the conditions on $(N_p)$
which appear in inequalities \eqref{eq:dfhcnp1} to \eqref{eq:dfhcnp7} are satisfied provided $(N_p)$ grows sufficiently fast and can be imposed at the very beginning of the proof.
\end{proof}

\begin{remark}
 Without any extra difficulty, we can modify the proof so that it works on any $\ell^p(\mathbb N)$, $p\in [1,+\infty)$, or on $c_0(\mathbb N)$.
\end{remark}

\subsection{Weakly partition regular pseudo-shifts}

The construction of frequently hypercyclic vectors or disjointly frequently hypercyclic vectors requires to exhibit
subsets of $\NN$ with positive lower density and satisfying some separation properties:
see e.g. \cite[Lemma 6.19]{BM09}, \cite[Theorem 5.15]{BCPalg} or \cite[Lemma 2.2]{MMP22}.
This leads the authors of \cite{MMP22} to impose a technical condition on the functions $(f_s)_{s=1,\dots,N}$ to state
their criterion for disjoint frequent hypercyclicity of a family $(T_{f_s,w_s})_{s=1,\dots,N}$ of pseudo-shifts: 
the functions $(f_s)_{s=1,\dots,N}$ have to be weakly partition regular.

\begin{definition}\label{def:wrp}
Let $f_s:\mathbb N\to\mathbb N$ be increasing maps, $s=1,\dots,N$ and let $\mathcal A$ be a family of subsets of $\mathbb N$. The family $\mathcal A$ is called $(f_s)_{s=1}^N-$\emph{weakly partition regular} is for any sequence $(B_j)_{j\geq 1}\subset\mathcal A$  there exists a sequence $(A_j)_{j\geq 1}\subset\mathcal A$ such that, for every $j,l\geq 1$, $A_j\subset B_j$ and 
\begin{equation}\label{eq:wpr}
f_s^m(\{1,\dots,j\})\cap f_t^n(\{1,\dots,l\})=\varnothing
\end{equation}
for all $m\in A_j$, for all $n\in A_l$, $m\neq n$, for all $s,t\in\{1,\dots,N\}.$
\end{definition}

We denote by $\Ald$ (resp. $\Aud$, $\Aubd$) the collection of subsets of $\NN$ with positive lower density (resp. upper density, resp. upper Banach density).
If $A\subset \mathbb N$, $\ldens(A)$ and $\udens(A)$ will mean respectively the lower density of $A$ and its upper density.

In \cite{MMP22}, it is shown that if $f_s(n)=n+q_s$ for some $q_s\geq 1$, which corresponds to the case of powers of weighted shifts, then $\Ald$ is $(f_s)$-weakly partition regular. 
If $q_s=1$, this is nothing else than \cite[Lemma 6.19]{BM09} which we have already used two times.
We now show that this remains true for all functions $(f_s)_{s=1,\dots,N}$. Hence, for studying disjoint frequent hypercyclicity of pseudo-shifts, we can dispense with the assumption ``$\Ald$ is $(f_s)$-weakly partition regular'' in the statement of Theorems 2.5 and 2.6 of \cite{MMP22}.

\begin{theorem}\label{thm:wrp}
Given any finite sequence $(f_s)_{s=1,\dots,N}$ of increasing maps $\mathbb N\to\mathbb N$, $\Ald$ is $(f_s)$-weakly partition regular.
\end{theorem}

The proof will be divided into several lemmas. 

\begin{lemma}
Let $B\subset\NN$ with $\ldens(B)>0$, let $h:\NN\to\mathbb R$ be increasing and let $\delta>0$. There exists $B'\subset B$ with $\ldens(B')>0$ such that $\udens(h(B')\cap\NN)<\delta$.
\end{lemma}
\begin{proof}
Let $(b_n)$ be an increasing enumeration of $B$. We set 
\begin{align*}
B_0&=\{b\in B:\ h(b)\notin \NN\}\\
B_1&=\{b\in B:\ h(b)\in \NN\} 
\end{align*}
and let $(b_{1,n})$ be the increasing enumeration of $B_1$. Observe that $h(b_{1,n})\geq n$ for all $n$.
Next we set $B'=B_0\cup\{b_{1,Kn}:\ n\in\NN\}$ for some large $K$. Then one shows by induction that, denoting $(b'_n)$ the increasing enumeration of $B'$, $b'_n\leq b_{Kn}$ for all $n\geq 1$. Moreover, 
$h(B')\cap\NN=\{h(b_{1,Kn}):\ n\in\NN\}=:\{c_n:\ n\geq 1\}$ where $(c_n)$ is increasing. 
Then $c_n\geq Kn$ so that $\udens(h(B')\cap\NN)\leq 1/K$. 
\end{proof}
By the inequality
$\udens(A\cup B)\leq \udens(A)+\udens(B)$, the lemma can be extended to a finite set of functions. 

\begin{lemma}\label{lem:wrp2}
Let $\mathcal F$ be a finite set of increasing functions from $\NN$ to $\RR$. Let $B\subset\NN$ with $\ldens(B)>0$ and let $\delta>0$. There exists $B'\subset\NN$ with
$\ldens(B')>0$ and $\udens(\bigcup_{h\in\mathcal F}h(B')\cap \NN)<\delta$ for all 
$h\in\mathcal F$.
\end{lemma}

The next lemma is the main step to prove that \eqref{eq:wpr} holds true for $j\neq l$.
\begin{lemma}\label{lem:wrp3}
Let $(\mathcal E_j)_{j\geq 1}$ be a sequence of finite sets of increasing functions from $\NN$ to $\RR$ going to $+\infty$. Let $(B_j)_{j\geq 1}\subset\Ald$. Then there exists a sequence $(A_j)_{j\geq 1}\subset\Ald$ such that 
\begin{enumerate}[(a)]
\item $\forall j\geq 1,\ A_j\subset B_j$;
\item $\forall j,l\geq 1$, $j\neq l$, $\forall F\in \mathcal E_j$, $\forall G\in \mathcal E_l$, $\forall m\in A_j$, $\forall n\in A_l$, $F(m)\neq G(n)$.
\end{enumerate}
\end{lemma}

\begin{proof}
Extending affinely each $F\in \bigcup_j \mathcal E_j$ between two consecutive integers, 
and on the left of $0$, we may assume that it defines a bijection $F:\mathbb R\to\mathbb R$. 
We shall denote by $\mathcal F_{j,l}=\{F^{-1}\circ G:\ \NN\to\RR,\ F\in \mathcal E_j,\ G\in\mathcal E_l\}$.
We will construct the sequence $(A_j)_{j\geq 1}$ inductively. For notational convenience, 
define $B_l(1)=B_l$ for $l\geq 1$. For $l\geq 2$, we apply Lemma \ref{lem:wrp2} to $B=B_l(1)$ and to $\mathcal F=\mathcal F_{1,l}$ to construct a set $B_l(2)$
such that 
\begin{itemize}
\item $B_l(2)\subset B_l(1)$;
\item $\ldens(B_l(2))>0$;
\item $\udens(\bigcup_{h\in\mathcal F_{1,l}}h(B_l(2))\cap\NN)<\frac{\ldens(B_{1}(1))}{4^l}$. 
\end{itemize}
By induction on $j\geq 1$, we construct sets $B_l(j)$, $l\geq j$, such that, for all $l\geq j+1$, 
\begin{itemize}
\item $B_l(j+1)\subset B_l(j)$;
\item $\ldens (B_l(j+1))>0$;
\item $\udens(\bigcup_{h\in \mathcal F_{j,l}}h(B_l(j+1))\cap \NN)<\frac{\ldens(B_j(j))}{4^l}.$
\end{itemize}
We finally set for $j\geq 1$ $A_j=B_j(j)\backslash \bigcup_{l\geq j+1}\bigcup_{h\in \mathcal F_{j,l}}h(B_l(j+1))\cap\NN$ which has positive lower density. Now let $j<l$, $F\in\mathcal E_j$ and $G\in\mathcal E_l$, $m\in A_j$, $n\in A_l$. The equality $F(m)=G(n)$ would imply that 
$m\in A_j\cap h(A_l)\cap\NN$ for $h=F^{-1}\circ G\in \mathcal F_{j,l}$. But this is impossible since 
$A_l\subset B_l(l)\subset B_l(j+1)$ and $A_j\cap h(B_l(j+1))=\varnothing$.
\end{proof}

We now handle the case $j=l$ with the following lemma.
\begin{lemma}\label{lem:wrp4}
Let $A\subset\NN$ with $\ldens(A)>0$. Let $f:\NN\to\RR$ be one-to-one. There exists $A'\subset A$ with $\ldens(A')>0$ such that, for all $n,m\in A'$, $n\neq m\implies n\neq f(m)$. 
\end{lemma}
\begin{proof}
Let $(a_n)$ be an increasing enumeration of $A$. We define by induction a sequence $(n_k)$ with $n_1=1$ and 
$$n_{k+1}=\inf\big\{n:\ a_n\notin f(\{a_1,\dots,a_{n_k}\})\cup f^{-1}(\{a_1,\dots,a_{n_k}\})\big\}.$$
Then $n_{k+1}\leq 2k+1$ so that, setting $A'=\{a_{n_k}:\ k\geq 1\}$, $\ldens(A')>0$.
Now for $n\neq m\in A'$, write $n=a_{n_k}$ and $m=a_{n_j}$ with $k\neq j$. If $k<j$, then one knows that $a_{n_j}\notin f^{-1}(\{a_{n_k}\})$ so that $f(m)\neq n$. If $k>j$, one knows directly that $n=a_{n_k}\neq f(a_{n_j})=m.$
\end{proof}

By successive extractions, we can extend the previous lemma to $f$ belonging to a finite set of functions $\mathcal F$. Applying this to the sets $(A_l)$ constructed in Lemma \ref{lem:wrp3} with $\mathcal F=\{F^{-1}\circ G:\ F,G\in\mathcal E_l\}$, we finally get 
\begin{lemma}\label{lem:wrp5}
Let $(\mathcal E_j)_{j\geq 1}$ be a sequence of finite sets of increasing functions from $\NN$ to $\RR$ going to $+\infty$. Let $(B_j)_{j\geq 1}\subset\Ald$. Then there exists a sequence $(A_j)_{j\geq 1}\subset\Ald$ such that 
\begin{enumerate}[(a)]
\item $\forall j\geq 1,\ A_j\subset B_j$;
\item $\forall j, l\geq 1$, $\forall F\in \mathcal E_j$, $\forall G\in \mathcal E_l$, $\forall m\in A_j$, $\forall n\in A_l$, $m\neq n\implies F(m)\neq G(n)$.
\end{enumerate}
\end{lemma}

\begin{proof}[Proof of Theorem \ref{thm:wrp}]
Let $j\geq 1$. Define, for $p\in\{1,\dots,j\}$ and $s\in\{1,\dots,N\}$, 
\begin{align*}
F_{p,s}:\mathbb N&\to\NN\\
m&\mapsto f_s^m(p)
\end{align*}
and let $\mathcal E_j=\left\{F_{p,s}:\ p\in\{1,\dots,j\},\ s\in\{1,\dots,N\}\right\}$. We apply Lemma \ref{lem:wrp5} to $(\mathcal E_j)$ and $(B_j)$.
We get a sequence of sets $(A_j)$ with positive lower density and $A_j\subset B_j$. Pick $j,l\geq 1$, $m\in A_j$, $n\in A_l$, and let $p\in\{1,\dots,j\}$, $q\in\{1,\dots,l\}$. Then
$$f_s^m(p)=F_{p,s}(m)\textrm{ whereas }f_t^n(q)=F_{q,t}(n)$$
and these two values are different.
\end{proof}

Theorem \ref{thm:wrp} remains valid if we replace $\Ald$ by  $\Aud$. The only change in the above proof is that we have to replace $\ldens(A\backslash B)\geq \ldens(A)-\udens(B)$ by $\udens(A\backslash B)\geq \udens(A)-\udens(B)$.
Nevertheless, this is not very useful since, for disjoint upper frequent hypercyclicity and disjoint reiterative hypercyclicity, simpler characterizations for pseudo-shifts are given in \cite{MMP22}.
These characterizations need a weaker notion than that of Definition \ref{def:wrp}.

\begin{definition}\label{def:arp}
Let $f_s:\mathbb N\to\mathbb N$ be increasing maps, $s=1,\dots,N$ and let $\mathcal A$ be a family of subsets of $\mathbb N$. The family $\mathcal A$ is called $(f_s)_{s=1}^N-$\emph{almost partition regular} is for any $l\geq 1$, any $B\in\mathcal A$,  there exists $A\subset B$, $A\in\mathcal A$ such that 
\begin{equation*}
f_s^m(\{1,\dots,l\})\cap f_t^n(\{1,\dots,l\})=\varnothing
\end{equation*}
$\textrm{ for all }m,n\in A,\ m\neq n, \textrm{ for all }s,t\in\{1,\dots,N\}.$
\end{definition}

This notion is weaker than being $(f_s)_{s=1}^N-$weakly partition regular. Therefore, the set $\Aud$ of subsets of $\NN$ with positive upper density is $(f_s)_{s=1}^N$-almost partition regular for any finite sequence $(f_s)_{s=1,\dots,N}$ of increasing maps $\mathbb N\to\mathbb N$. To also delete the assumption $\Aubd$ is $(f_s)$-almost partition regular
in Theorem 3.5 of \cite{MMP22}, we need the following result whose proof follows from a suitable modification of Lemma \ref{lem:wrp4}.
The only change is to observe that if $\{n_k:\ k\geq 1\}$ has positive upper Banach density and if $(p_k)_{k\geq 1}$ is such that $p_k\leq Ck$ for some $C>0$ and all $k\geq 1$,
then $\{n_{p_k}:\ k\geq 1\}$ has positive upper Banach density.

\begin{proposition}
Given any finite sequence $(f_s)_{s=1,\dots,N}$ of increasing maps $\mathbb N\to\mathbb N$, $\Aubd$ is $(f_s)$-almost partition regular.
\end{proposition}

Lemma \ref{lem:wrp5} allows us to obtain a far reaching extension of \cite[Lemma 6.29]{BM09} and of \cite[Lemma 2.2]{MMP22}.
This extension will be used later on in this paper and has potentially other applications for the study of frequently hypercyclic operators.

\begin{theorem}\label{thm:sets}
 Let $\mathcal F$ be a finite set of increasing functions from $\NN$ to $\RR$ going to $+\infty$.
 Let $(B_k)_{k\geq 1}\subset\Ald$ and let $(N_k)_{k\geq 1}$ be a sequence of positive integers. Then
 there exists a sequence $(A_k)_{k\geq 1}\subset\Ald$ such that
 \begin{enumerate}[(a)]
  \item $\forall k\geq 1$, $A_k\subset B_k$, $\min(A_k)\geq N_k$;
  \item $\forall k,l\geq 1$, $\forall (n,m)\in A_k\times A_l$, $\forall (F,G)\in\mathcal F^2$, 
  $n\neq m\implies |F(n)-G(m)|\geq N_k+N_l.$
 \end{enumerate}
\end{theorem}
\begin{proof}
 For $f\in\mathcal F$, $j\in\ZZ$, define $f_j(n)=f(n)+j$ and set $\mathcal E_k=\{f_j:\ f\in\mathcal F,\ |j|\leq N_k\}$. 
 We apply Lemma \ref{lem:wrp5} to $(\mathcal E_k)_k$ and $B'_k$ with $B'_k=B_k\cap[N_k;+\infty)$
 and we conclude as in the proof of Theorem \ref{thm:wrp}.
\end{proof}

This theorem will be used in the following form. 

\begin{corollary}\label{cor:sets}
 Let $\phi_1,\phi_2:\NN\to\RR$ be two increasing functions tending to $+\infty$ and such that $\phi_2-\phi_1$ tends to $+\infty$. 
 Let $(N_k)_k$ be a sequence of positive integers and let $(B_k)\subset\Ald$. There exists $(A_k)\subset \Ald$ such that, for all $k,l\geq 1$, $\min(A_k)\geq N_k$,
 $A_k\subset B_k$ 
 and for every $(n,m)\in A_k\times A_l$, for every $(i,j)\in\{1,2\}^2$, 
 \begin{equation}\label{eq:corsets1}
 (i,n)\neq (j,m)\implies |\phi_i(n)-\phi_j(m)|\geq N_k+N_l.
 \end{equation}
\end{corollary}
\begin{proof}
 For all $k\geq 1$, let $M_k\geq N_k$ be such that $n\geq M_k$ implies $\phi_2(n)-\phi_1(n)\geq 2N_k$. We apply Theorem \ref{thm:sets}
 to $\mathcal F=\{\phi_1,\phi_2\}$, $B'_k=[M_k;+\infty)\cap B_k$ and $(M_k)$. The result follows immediately.
\end{proof}

Let us show an application of this corollary which will be used in Section \ref{sec:cohol}.
\begin{corollary}\label{cor:setscoholpar}
 Let $(a_i,b_i)\in\CC^*\times\CC$, $i=1,2$ with $(a_1,b_1)\neq (a_2,b_2)$. Let $(N_k)$ be a sequence
 of integers. There exists $(A_k)\subset\Ald$ satisfying, for all $k,l\geq 1$, $\min(A_k)\geq N_k$ and,
 for every $(n,m)\in A_k\times A_l$, for every $(i,j)\in\{1,2\}^2$, 
 \begin{eqnarray}
 \label{eq:corsetscoholpar1}
  (i,n)\neq (j,m)&\implies&\big|a_i n+b_i\log n-a_j n-b_j\log n\big|\geq N_k+N_l. 
 \end{eqnarray}
\end{corollary}
\begin{proof}
 Using $|z|\geq\max\big(\Re e(z),\Im m(z)\big)$, we may and shall assume that $a_1$ and $a_2$ are distinct real numbers and that $b_1$ and $b_2$
 are real. Without loss of generality, we also assume that $|a_1|,|a_2|>1$. If $a_1$ and $a_2$ have opposite signs, say $a_1<0<a_2$, then 
 for all $n,m\in\NN$, 
 $$|a_1 n+b_1\log n-a_2 m-b_2\log m|\geq |a_1|\left(n-\frac{|b_1|}{|a_1|}\log n\right)+|a_2|\left(m-\frac{|b_1|}{|a_1|}\log m\right).$$
 Therefore, \eqref{eq:corsetscoholpar1} for $i\neq j$ is satisfied provided $n\geq M_k$
 and $m\geq M_l$ for some sequence $(M_k)$. To ensure \eqref{eq:corsetscoholpar1} for $i=j$, we set 
 \begin{align*}
  \phi_1(n)&=-\big( a_1n+b_1\log n\big)\\
  \phi_2(n)&= a_2n+b_2\log n
 \end{align*}
 which are eventually increasing. We also set $\mathcal F=\{\phi_1,\phi_2\}$, $B_k=[M_k,+\infty)\cap\NN$
 and we apply Theorem \ref{thm:sets} to $\mathcal F$, $(B_k)$ and $(N_k)$ to get a sequence $(A_k)\subset\Ald$. 
 This ends up the proof in that case since, for $i=1,2$, $(n,m)\in\NN^2$, 
 $$\big| a_i n+b_i\log n-a_i m-b_i\log m\big|\geq |\phi_i(n)-\phi_i(m)|-1.$$
 Assume now that $1<a_1,a_2$ (if $a_1=a_2$, then $b_1\neq b_2$) and set
  \begin{align*}
  \phi_1(n)&= a_1n+b_1\log n\\
  \phi_2(n)&= a_2n+b_2\log n
 \end{align*}
Then the result follows directly from Corollary \ref{cor:sets}.
\end{proof}


\section{Composition operators on $H^2(\mathbb D)$}

\subsection{Preliminaries}

We recall that $\lfm$ is the set of all linear fractional maps of $\DD$ and we denote by $\aut$ its subset of all automorphisms.
Linear fractional maps of $\DD$ may be classified into three classes:
\begin{itemize}
 \item parabolic members of $\lfm $ have only one fixed point, which lies on $\TT$ (and is necessarily attractive); we will denote by $\para$ the set of all parabolic linear fractional maps of $\DD$.
 \item hyperbolic members of $\lfm $ have one attractive fixed point on $\TT$ and a repulsive fixed point outside $\DD$; the second fixed point belongs to $\TT$
 if and only if the map is an automorphism; we will denote by $\hyp$ the set of all hyperbolic linear fractional maps of $\DD$.
 \item elliptic and loxodromic members of $\lfm $ have an attractive fixed point in $\DD$ and a repulsive fixed point outside $\overline{\DD}$.
\end{itemize}

Only the first two classes give rise to composition operators on $H^2(\DD)$ with interesting dynamical properties.
Therefore, we may and shall pick $\varphi\in \lfm $ parabolic or hyperbolic. 
Upon conjugation by a suitable automorphism, it will always be possible to assume that the attractive fixed point is $+1$
and that the repulsive fixed point (if any) lies on the line $(-\infty,-1]$. These maps will be more clearly understood by moving to the right half-plane 
$\CC_0$, where $\CC_a$ means $\{w:\ \Re e(w)>a\}$. Let $\mathcal C(z)=\frac{1+z}{1-z}$ be the Cayley map which maps the unit disc onto $\CC_0$. 
Let also $\mathcal H^2=\{f\circ \mathcal C^{-1}:\ f\in H^2(\mathbb D)\}$ which is endowed with
$$\|F\|_2^2 =\pi^{-1}\int_{\mathbb R}|F(t)|^2 dt/(1+t^2).$$
We will denote by $\minv$ the measure $\frac{dt}{\pi(1+t^2)}$. 

Via $\mathcal C$, the model map of a hyperbolic linear fractional map $\varphi$ of $\DD$ is a dilation $\psi(w)=\lambda (w-b)+b$, with $\lambda=1/\varphi'(1)>1$ and $\Re e(b)\leq 0$
(with the above normalization we may assume $b\in (-\infty,0]$ and even $b\in (-1,0)$ by conjugating $\psi$ with some dilation $w\mapsto \mu w$).
The model map of a parabolic automorphism $\varphi$ of $\DD$ is a vertical translation $\psi(w)=w+i\tau$ for some $\tau\in\RR^*$. 
Since we simultaneously deal with two linear fractional maps $\varphi_1$ and $\varphi_2$, if $\varphi_1$ and $\varphi_2$ do not have the same attractive
fixed point, it will not be possible to ensure that their model maps on $\CC_0$ both have this reduced form. In that case, we may assume that the
respective attractive fixed points are $+1$ and $-1$. In particular, if $\varphi\in\para\cap\aut$ has $-1$ as fixed point, then $\psi=\mathcal C\circ\varphi\circ\mathcal C^{-1}$
can be written $\psi(w)=\frac{w}{1+i\tau w}$ for some $\tau\in\RR^*$ (this comes from $w\mapsto w+i\tau$ conjugated by $w\mapsto 1/w$).

Let $\varphi\in\hyp$. Then $\Delta=\bigcup_{n\geq 1}\varphi^{-n}(\mathbb D)$ is a disc which contains the unit disc $\DD$ and $\varphi$ is an automorphism of $\Delta$. We will call $\Delta$ 
the automorphism domain of $\varphi$. Observe that the repulsive fixed point of $\varphi$ is nothing else than the attractive fixed point of $\varphi^{-1}$.
If $\varphi\in\para\cap\aut$, then we define its automorphism domain as the unit disk $\DD$.

\subsection{Geometric behaviour of linear fractional maps}
%

The proof of Theorem \ref{thm:coh2} will need to finely analyze the behaviour of the iterates of a linear fractional map and of its reciprocal on the circle. We start  with two geometric lemmas on the behaviour of the reciprocal of a linear fractional map far away from its attractive fixed point.

\begin{lemma}\label{lem:coh2hypdifferent1}
Let $\varphi\in\hyp\cup(\para\cap\aut)$, $\alpha\in\TT$ its attractive fixed point, $\delta_0>0$. Then there exists $C>0$ such that, for all $\xi\in \TT\backslash D(\alpha,2\delta_0)$, for all $\delta\in(0,\delta_0)$, for all $l\geq 0$, 
\begin{itemize}
\item $\varphi^{-l}\big(D(\xi,\delta)\cap\overline{\DD}\big)\subset D(\varphi^{-l}(\xi),C\delta\lambda^l)$ when $\varphi$ is hyperbolic and $\lambda=\varphi'(\alpha)$;
\item $\varphi^{-l}\big(D(\xi,\delta)\cap\overline{\DD}\big)\subset D(\varphi^{-l}(\xi),C \delta l^{-2})$
when $\varphi$ is parabolic.
\end{itemize}
\end{lemma}
\begin{proof}
Assume first that $\varphi$ is hyperbolic with attractive fixed point $+1$. Its half-plane model is the map
$\psi(w)=\lambda^{-1}(w-b)+b$ where we can assume $\Re e(b) \in(0,1)$. Let $it_0=\mathcal C(\xi)$. There exists $C_1>0$
which does not depend on $\delta\in(0,\delta_0)$ and on $\xi\in\TT\backslash D(1,2\delta_0)$ such that $\mathcal C(D(\xi,\delta)\cap\overline{\DD})\subset D(it_0,C_1\delta)\cap\overline{\CC_0}$. Now, for all $w\in D(it_0,C_1\delta)$, 
$$|\psi^{-l}(w)-\psi^{-l}(it_0)|=\lambda^l |w-it_0|\leq \lambda^l C_1\delta.$$
We get the result by coming back to the disc with $\mathcal C^{-1}$.

Suppose now that $\varphi$ is a parabolic automorphism and let us assume that its attractive fixed point is $1$. Then $\varphi$ and $\varphi^{-l}$ can be written respectively 
$$\varphi(z)=1+\frac{2(z-1)}{2-i\tau(z-1)},\ \varphi^{-l}(z)=1+\frac{2(z-1)}{2+il\tau(z-1)},\ \tau\in\mathbb R^*$$
(see \cite[Chapter 0]{Sha93}). Therefore, for $z\in D(\xi,\delta)\cap \overline{\DD}$, 
\begin{equation}\label{eq:lemcoh2hypdifferent1}
\big|\varphi^{-l}(z)-\varphi^{-l}(\xi)\big|=\frac{4|z-\xi|}{|2+il\tau(\xi-1)|\cdot |2+il\tau(z-1)|}.
\end{equation}
By definition of $\delta_0$, there exists $\eta>0$ (which does not depend on $\delta$ or $\xi$) such that $\Re e(1-z)>\eta$ and $\Re e(1-\xi)>\eta$. Equality \eqref{eq:lemcoh2hypdifferent1} then yields immediately the result. 
\end{proof}

\begin{lemma}\label{lem:coh2hypdifferent2}
Let $\varphi\in\hyp$, let $\alpha,\ \beta$ be respectively the attractive and repulsive fixed points of $\varphi$, $\lambda=\varphi'(\alpha)$. Then there exists $C>0$ such that, for all $l\geq 1$, 
$$w\in \overline{\DD}\backslash D(\alpha,\lambda^{l/2})\implies \varphi^{-l}(w)\in D(\beta,C\lambda^{l/2}).$$
\end{lemma}
\begin{proof}
We assume $\alpha=1$, we move on $\overline{\CC_0}$ and we consider $\psi(w)=\lambda^{-1}(w-b)+b$ with $b\in (-1,0)$. Then 
$$\mathcal C(\overline{\DD}\backslash D(1,\lambda^{l/2}))\subset \{w\in\overline{\CC_0}:\ |w|\leq 2\lambda^{-l/2}\}$$
and we just need to prove that $\psi^{-l}(w)$ is contained in $D(b,C\lambda^{l/2})$ for some $C>0$ for all $w\in\overline{\CC_0}$ with $|w|\leq 2\lambda^{-l/2}$.
Pick such a $w$ and observe that $\psi^{-l}(w)=\lambda^{l}(w-b)+b$
so that 
$$|\psi^{-l}(w)-b|\leq \lambda^{l}(2\lambda^{-l/2}+|b|)\leq C \lambda^{l/2}.$$
\end{proof}

We will also need a lemma which quantifies the measure of the set of points $z\in\TT$ such that $\varphi^{l}(z)$ is far from the attractive fixed point of $\varphi\in\hyp$. 
\begin{lemma}\label{lem:coh2hypdifferent4}
Let $\varphi\in\hyp$, $\alpha$ its attractive fixed point, $\lambda=\varphi'(\alpha)$. Then there exists $C>0$ such that, for all $\delta\in(0,1)$, for all $k\geq 1$, 
$$\textrm{meas}\left(\left\{z\in\TT:\ \varphi^{k}(z)\notin D(\alpha,\delta)\cap\overline{\DD}\right\}\right)\leq \frac{C}\delta \lambda^k.$$
\end{lemma}
\begin{proof}
Again, we move on $\CC_0$ and we consider $\psi(w)=\lambda^{-1}(w-b)+b$ with $b\in(-1,0)$.
 Geometric considerations show that there exists $C>0$ such that for all $\delta\in(0,1)$,
$$\mathcal C\big(D(1,\delta)\cap\overline{\DD}\big)\supset \big\{w\in\overline{\CC_0}:\ |w|\geq C\delta^{-1}\big\}.$$
Therefore, it is sufficient to show that 
$$\minv\big(\big\{t\in\RR:\ |\psi^{k}(it)|\leq C\delta^{-1}\big\}\big)\leq \frac C\delta\lambda^k.$$
Now, $|\psi^{k}(it)|\geq \lambda^{-k}|t|$ so that 
$$\big\{t\in\RR:\ |\psi^{k}(it)|\leq C\delta^{-1}\big\}\subset \big\{t\in\RR:\ |t|\leq C\delta^{-1}\lambda^k\big\}$$
and it is clear that this last set has invariant measure less than $C\delta^{-1}\lambda^k$.
\end{proof}

A similar result holds true for parabolic automorphisms.
\begin{lemma}\label{lem:coh2pardifferent}
Let $\varphi\in\aut\cap\para$, $\alpha$ its attractive fixed point. There exist $C_1,C_2>0$ such that, for all $k\geq 1$, 
$$meas\left(\left\{z\in\TT:\ \varphi^k (z)\notin D(\alpha, C_1 k^{-1})\right\}\right)\leq C_2 k^{-1}.$$
\end{lemma}
\begin{proof}
We assume $\alpha=1$ and move on the half-plane by considering $\psi(w)=w+i\tau$, $\tau>0$. Let $C_0$ be such that, for all $\delta>0$,
$$\mathcal C\big(D(1,\delta)\cap\TT\big)\supset \{it:\ |t|\geq C_0\delta^{-1}\}.$$
Let  $C_1>0$ be such that $C_0C_1^{-1}<\tau$. We are reduced to prove that 
\begin{equation}\label{eq:coh2pardifferent2}\minv \left(\left\{t\in\mathbb R:\ |t+k\tau|\leq C_0C_1^{-1}k\right\}\right)\leq C_2k^{-1}
\end{equation}
for some $C_2>0$. Now 
$$|t+k\tau|\leq C_0C_1^{-1} k\iff t\in [-k\tau-C_0C_1^{-1}k,-k\tau+C_0C_1^{-1}k]$$
and \eqref{eq:coh2pardifferent2} holds by our choice of $C_1$.
\end{proof}

\subsection{Several dense sets of functions in $H^2(\DD)$}

To apply Theorem \ref{thm:dfhcc}, we will need to exhibit dense subsets of $H^2(\DD)$ which are well adapted to the two self-maps of $\DD$ which come into play. The first lemma, which belongs to folklore, was already used in the context of linear dynamics of composition operators (see \cite[Lemma 1.48]{BM09} or \cite[Lemma 5]{Ta04}).

\begin{lemma}\label{lem:denseh2easy}
Let $d\geq 1$, $N\geq 1$, $\alpha_1,\dots,\alpha_N\in\CC\backslash \DD$. The set of holomorphic polynomials with a zero of order not less than $d$ at each $\alpha_i$ is dense in $H^2(\DD)$. 
\end{lemma}

We will need a much more subtle dense subset of $H^2(\DD)$ which is well adapted to the dynamics of  the reciprocal of a linear fractional map. 
\begin{lemma}\label{lem:coh2hypdifferent3}
Let $\varphi\in\hyp\cup(\para\cap\aut)$ and let $\Delta$ be its automorphism domain. Let $\alpha$ (resp. $\beta$) be the attractive fixed point of $\varphi$ (resp. $\varphi^{-1}$),
$\lambda=\varphi'(\alpha)$. Let $\gamma\in \TT\backslash\{\alpha,\beta\}$ and
let us define $\gamma_l=\varphi^{-l}(\gamma)$, $l\geq 1$. Let also $d\geq 1$,  $N\geq 0$ and $\xi_1,\dots,\xi_N\in \CC\backslash\DD$.
Then the set of functions $f:\overline{\Delta}\backslash\{\beta\}\to \CC$ such that $f\in H^\infty(\Delta)$ and there exists $M>0$ satisfying
\begin{enumerate}[(a)]
\item for all $l\geq 1$, for all $z\in\overline{\Delta}\backslash\{\beta\}$, 
$$|f(z)|\leq M\left\{
\begin{array}{ll}
l^{2d}|z-\gamma_l|^d&\textrm{ if $\varphi$ is an automorphism}\\
\lambda^{-dl}|z-\gamma_l|^d&\textrm{ otherwise,}
\end{array}\right.$$
\item for all $i=1,\dots,N$, for all 
$z\in\overline{\Delta}\backslash\{\beta\}$, $|f(z)|\leq M|z-\xi_i|^d$,
\end{enumerate}
is dense in $H^2(\DD)$.
\end{lemma}
\begin{proof}
We first assume that $\varphi$ is not an automorphism, so that $\gamma_l\in\Delta\backslash\DD$
for all $l\geq 1$. The sequence $(\gamma_l)$ is a Blaschke sequence with respect to $\Delta$.
Let $B$ be the Blaschke product (of $\Delta$) whose zeros are the $\gamma_l$, $l\geq 1$, with multiplicity $d$.
We claim that for any polynomial $P$ with a zero of order at least $d$ at $\xi_i$, $i=1,\dots,N$,
the function $f=BP$ satisfies (a) and (b). The last point is immediate. To prove (a), we may assume that $\Delta$ is the disk $D(-\rho,1+\rho)$ for some $\rho>0$
namely that $\beta=-1-2\rho$. $B$ is then given by $\prod_{l\geq 1}\left(\frac{T(z)-T(\gamma_l)}{1-\overline{T(\gamma_l)}T(z)}\right)^d$ where $T:\Delta\to\DD,\ z\mapsto (z+\rho)/1+\rho$.
Each factor has modulus less than $1$ on $\Delta$. Moreover, since $(\gamma_l)$ tends nontangentially (with respect to $\Delta$) to $\beta$,
we know that there exists $C_1>0$ such that 
\begin{align*}
 1+\rho-|\gamma_l+\rho| &\geq C_1|\beta-\gamma_l|\\
 &\geq C_2\lambda^{l}
\end{align*}
for some $C_2>0$. This yields easily the existence of $M>0$ such that, for all $l\geq 1$ and all $z\in\Delta$,
$$\frac1{|1-\overline{T(\gamma_l)}T(z)|}\leq M\lambda^{-l}$$
which in turn proves (a).

Moreover, the set $\{BP:\ P\ $ is a polynomial with a zero of order at least $d$ at $\xi_i$, $i=1,\dots,N\}$ is dense in $H^2(\DD)$. Indeed, pick $f\in H^2$, $\veps>0$ and let $P$ be such a polynomial with $\left\|P-\frac fB\right\|_2<\veps$ (observe that $|B|\geq \delta>0$ on $\DD$). Then $\|BP-f\|_2<\veps$. 

We turn to the case of automorphisms. For $a\in(0,1)$, let 
$$B_a(z)=\prod_{l=1}^{+\infty}\left(\frac{z-\gamma_l}{z-\left(1+\frac a{l^2}\right)\gamma_l}\right)^d.$$
Let $z\in\overline{\DD}\backslash\{\beta\}$. Then there exists $l_0\geq 1$ and $\delta>0$ such that, for all $l\geq l_0$, $|z-\gamma_l|>\delta$. We write
$$\frac{z-\gamma_l}{z-\left(1+\frac a{l^2}\right)\gamma_l}=1+\frac a{l^2}\times\frac{\gamma_l}{z-\left(1+\frac a{l^2}\right)\gamma_l}.$$
Since for $l\geq l_0$, 
\begin{equation}\label{eq:lemcoh2hypdifferent3}
\left| \frac a{l^2}\times\frac{\gamma_l}{z-\left(1+\frac a{l^2}\right)\gamma_l}\right|\leq \frac 1{l^2 |z-\gamma_l|}\leq \frac1{\delta l^2},
\end{equation}
the infinite product converges on $\overline{\DD}\backslash\{\beta\}$. It is not difficult to see that the convergence is locally uniform on $\overline{\DD}\backslash\{\beta\}$ so that $B_a\in H^\infty(\DD)$ with $\|B_a\|_\infty\leq 1$. Let $P$ be a holomorphic polynomial. We shall prove that $B_a P\xrightarrow{a\to 0}P$ in $H^2(\DD)$. It is sufficient to prove that $B_a(z)$ tends to $1$ for all $z\in \overline{\DD}\backslash\{\beta\}$ as $a$ tends to $0$. But this follows from Lebesgue's convergence theorem applied to $\sum_l \log\left(1+\frac a{l^2}\times\frac{\gamma_l}{z-\left(1+\frac a{l^2}\right)\gamma_l}\right)$, using \eqref{eq:lemcoh2hypdifferent3}. 
Finally, for all $z\in\overline{\DD}\backslash\{\beta\}$, for all $a\in(0,1)$,
$$|B_a(z)|\leq \frac{|z-\gamma_l|^d}{\left|z-\left(1+\frac a{l^2}\right)\gamma_l\right|^d}\leq \frac{l^{2d}}a|z-\gamma_l|^d.$$ 
Lemma \ref{lem:denseh2easy} ends up the proof.
\end{proof}

\subsection{Estimates of (sum of) functions in $H^2(\DD)$}

To prove the uniform unconditional convergence of the series that appear in Theorem \ref{thm:dfhcc}, we will need to estimate the norm of some sums of functions of $H^2(\DD)$. The case of hyperbolic maps $\varphi$ will be significantly easier since the convergence of
$\varphi^n$ to the attractive fixed point is so fast that it will ensure the convergence of $\sum_n \|C_\varphi^n(f)\|$ for well choosen functions $f$. What we need in this paper is contained in the two forthcoming lemmas. 

\begin{lemma}\label{lem:coh2hypsame}
Let $b_1,b_2\in\CC$ with $-1<\Re e(b_1),\ \Re e(b_2)\leq 0$ and let $\beta_1=\mathcal C^{-1}(b_1)$. Let $P$ be a polynomial with $P(1)=P(\beta_1)=0$ and let $F=P\circ \mathcal C^{-1}$. There exists $C>0$ such that
\begin{enumerate}[(a)]
\item for all $\gamma\geq 1$, for all $\kappa\in(0,1)$, 
$$\int_{\mathbb R} |F(\gamma(it-b_2)+\kappa (b_2-b_1)+b_1)|^2\frac{dt}{1+t^2}\leq \frac{C}{\gamma^{1/2}}.$$
\item for all $\gamma\in (0,1)$, for all $\kappa\in(0,1)$, 
$$\int_{\mathbb R} |F(\gamma(it-b_2)+\kappa (b_2-b_1)+b_1)|^2\frac{dt}{1+t^2}\leq {C}\big(\gamma^{1/2}+\kappa^2\big).$$
\end{enumerate}
\end{lemma}
\begin{proof}
Let $\sigma=\min(\Re e(b_1),\Re e(b_2))\in (-1,0]$. We observe that, for any $t\in\mathbb R$, any $\gamma>0$ and any $\kappa\in (0,1)$,
$\gamma(it-b_2)+\kappa (b_2-b_1)+b_1$ belongs to $\mathbb C_\sigma$. Indeed, if $\sigma=\Re e(b_1)$, 
\begin{align*}
\Re e\big(\gamma(it-b_2)+\kappa(b_2-b_1)+b_1\big)&= -\gamma \Re e(b_2)+\kappa \Re e(b_2-b_1)+\Re e(b_1)\\
&\geq \Re e(b_1)
\end{align*}
whereas, if $\sigma=\Re e(b_2)$, then 
\begin{align*}
\Re e\big(\gamma(it-b_2)+\kappa(b_2-b_1)+b_1\big)&= -\gamma \Re e(b_2)+(1-\kappa) \Re e(b_1-b_2)+\Re e(b_2)\\
&\geq \Re e(b_2).
\end{align*}
The set $\Omega=\mathcal C^{-1}(\mathbb C_\sigma)$ is a disk in $\mathbb C$ containing $\mathbb D$ and having $1$ as a boundary point. Our assumption on $P$ ensures the existence of $C_1>0$ such that, for all $z\in\bar{\Omega}$, 
$$|P(z)|\leq C_1|z-1|\textrm{ and }|P(z)|\leq C_1|z-\beta_1|.$$
This inequality transfers to $F$ as follows: for all $w$ in $\overline{\mathbb C_{\sigma}}$, 
$$|F(w)|\leq C_1\left|\frac{w-1}{w+1}-1\right|\leq \frac{2C_1}{|w+1|}\textrm{ and }
|F(w)|\leq C_1\left|\frac{w-1}{w+1}-\frac{b_1-1}{b_1+1}\right|\leq C_2 |w-b_1|.$$
We are now ready to prove (a). Let $\gamma\geq 1$, $\kappa\in(0,1)$ and $t\in\mathbb R$. Then 
$$|F(\gamma(it-b_2)+\kappa(b_2-b_1)+b_1)|^2\lesssim \frac{1}{(1+x'+\gamma x)^2+(y'+\gamma(t-y))^2}$$
for some $x,x',y,y'\in \mathbb R$ which do not depend on $\gamma$ (but $x'$ and $y'$ depend on $\kappa$) and with $x\geq 0$ and $x'>-1$. Therefore,
$$|F(\gamma(it-b_2)+\kappa(b_2-b_1)+b_1)|^2\lesssim \frac{1}{1+(y'+\gamma(t-y))^2}.$$
We first integrate on the set of $t$ such that $|\gamma(t-y)+y'|^2\geq \gamma$. 
Then 
$$\int_{|\gamma(t-y)+y'|^2\geq \gamma} |F(\gamma(it-b_2)+\kappa (b_2-b_1)+b_1)|^2\frac{dt}{1+t^2}\lesssim\int_{\mathbb R}\frac{dt}{\gamma(1+t^2)}\lesssim \frac 1\gamma.$$
On the other hand, the interval $|\gamma(t-y)+y'|^2\leq \gamma$ has Lebesgue measure $2\gamma^{-1/2}$. Since $F$ is bounded on $\mathbb C_\sigma$, we obtain
$$\int_{|\gamma(t-y)+y'|^2\leq \gamma} |F(\gamma(it-b_2)+\kappa (b_2-b_1)+b_1)|^2\frac{dt}{1+t^2}
\lesssim \frac 1{\gamma^{1/2}}.$$
Let us turn to (b). In that case, for all $\gamma\in(0,1)$, all $\kappa\in(0,1)$ and all $t\in\mathbb R$, 
\begin{align*}
|F(\gamma(it-b_2)+\kappa(b_2-b_1)+b_1)|^2&\lesssim |\gamma(it-b_2)+\kappa(b_2-b_1)|^2\\
&\lesssim \gamma^2 |t-y_1|^2+\kappa^2 y_2+\gamma^2 y_3
\end{align*}
with $y_1,y_2,y_3$ independent of $\kappa$ and $\gamma$. On the interval $|t-y_1|^2<\gamma^{-1}$, we use this inequality to get
$$\int_{|t-y_1|^2<\gamma^{-1}}|F(\gamma(it-b_2)+\kappa(b_2-b_1)+b_1)|^2\frac{dt}{1+t^2}\lesssim\int_{\mathbb R}( \gamma+\kappa^2)\frac{dt}{1+t^2}\lesssim \gamma+\kappa^2.$$
On the remaining part of $\mathbb R$, we just use that $F$ is bounded on $\mathbb C_\sigma$ to get
$$\int_{|t-y_1|^2>\gamma^{-1}} |F(\gamma(it-b_2)+\kappa(b_2-b_1)+b_1)|^2\lesssim \int_{|t-y_1|^2>\gamma^{-1}}\frac{dt}{1+t^2}\lesssim \gamma^{1/2}.$$
\end{proof}

The next lemma will play the role of Lemma \ref{lem:coh2hypsame} in the context of hyperbolic maps with different attractive fixed points.

\begin{lemma}\label{lem:coh2hypdifferent5}
Let $\varphi_1,\varphi_2\in\lfm$ with $\varphi_2$ hyperbolic.
Let $\alpha_1,\ \alpha_2$ be the respective attractive fixed points of $\varphi_1$ and $\varphi_2$, $\lambda=\varphi_2'(\alpha_2)$, 
and assume that $\alpha_1\neq\alpha_2$. Let $\Delta$ be the automorphism domain of $\varphi_2$ and let $\beta_2$ be the repulsive fixed point of $\varphi_2$. Let $f:\overline{\Delta}\backslash\{\beta_2\}\to\mathbb C$ be holomorphic in $\Delta$ and such that $|f(z)|\leq A |z-\beta_2|$ for some $A>0$, for all $z\in\overline{\Delta}\backslash\{\beta_2\}$. Then there exists $M>0$ such that, for all $l,k\geq 0$, 
$$\|f\circ \varphi_2^{-l}\circ\varphi_1^k\|\leq M \lambda^{l/2}.$$
\end{lemma}
\begin{proof}
We write 
$$\|f\circ\varphi_2^{-l}\circ\varphi_1^k\|^2=\int_0^{2\pi}|f\circ\varphi_2^{-l}(\varphi_1^k(e^{i\theta}))|^2\frac{d\theta}{2\pi}$$
and we denote by
$$A_{k,l}=\big\{\theta\in[0,2\pi]:\ \varphi_1^k(e^{i\theta})\in D(\alpha_2,\lambda^{l/2})\big\}.$$
Using Lemma \ref{lem:coh2hypdifferent2}, 
\begin{align*}
\int_{[0,2\pi]\backslash A_{k,l}} |f\circ\varphi_2^{-l}(\varphi_1^k(e^{i\theta}))|^2\frac{d\theta}{2\pi}&\leq \int_{[0,2\pi]\backslash A_{k,l}}A^2 |\varphi_2^{-l}(\varphi_1^k(e^{i\theta}))-\beta_2|^2\frac{d\theta}{2\pi}\\
&\leq A^2 C^2 \lambda^l.
\end{align*}
Thus since $f$ is bounded on $\overline{\DD}$, we are reduced to prove that 
there exists $C_0>0$ such that, for all $k,l$, 
\begin{equation} \label{eq:lemcoh2hypdifferent5}
 meas(A_{k,l})\leq C_0\lambda^{l/2}.
\end{equation}
 Let $l_0$ be such that $\alpha_1\notin D(\alpha_2,\lambda^{l_0})$ and 
let us require $C_0\geq \lambda^{-l_0/2}$ so that \eqref{eq:lemcoh2hypdifferent5} is true for all $l\leq l_0$ and all $k\geq 0$. 
It remains to handle the case $l\geq l_0$. Assume first that $\varphi_1$ is not an automorphism. 
Then there exists $k_0>0$ such that, for all $k\geq k_0$, $\varphi_1^k(\TT)\cap D(\alpha_2,\lambda^{l_0/2})$ is empty. Therefore, \eqref{eq:lemcoh2hypdifferent5} is true for $k\geq k_0$
and $l\geq l_0$. On the other hand, for a fixed $k$, it is clear that there exists some $d_k$ such that $A_{k,l}$ has measure less than $d_k \lambda^{l/2}$ for all $l\geq l_0$. 
We conclude by setting $C_0=\max(\lambda^{-l_0/2},d_0,\dots,d_{k_0})$.

Assume now that $\varphi_1$ is a hyperbolic automorphism with $\alpha_1=1$ and $\beta_1=-1$. We move on $\CC_0$
where the model map $\psi_1$ of $\varphi_1$ is given by $\psi_1(w)=\mu w$ for some $\mu>1$ and we want to evaluate 
$$\minv\left(\left\{t\in\RR:\ \mu^k t\in [a-\kappa \lambda^{l/2},a+\kappa \lambda^{l/2}]\right\}\right)$$
where $a\neq 0$ and $\kappa>0$ only depend on $\alpha_2$. Then 
\begin{align*}
\minv\left(\left\{t\in\RR:\ \mu^k t\in [a-\kappa \lambda^{l/2},a+\kappa \lambda^{l/2}]\right\}\right)&=\frac 1\pi\arctan\left(\frac{a+\kappa\lambda^{l/2}}{\mu^k}\right)-\frac 1\pi\arctan\left(\frac{a-\kappa\lambda^{l/2}}{\mu^k}\right)\\
&\leq C_1\lambda^{l/2}
\end{align*}
since $\mu>1$. Finally, if $\varphi_1$ is a parabolic automorphism with $\alpha_1=1$, we repeat the argument with $\psi_1(w)=w-i\tau$ with $\tau>0$. We are reduced to estimate
\begin{align*}
&\minv\left(\left\{t\in\RR:\ t-k\tau \in [a-\kappa \lambda^{l/2},a+\kappa \lambda^{l/2}]\right\}\right)\\
&\quad\quad =\frac 1\pi\arctan\left(a+\kappa\lambda^{l/2}+k\tau\right)-\frac 1\pi\arctan\left(a-\kappa\lambda^{l/2}+k\tau\right)\\
&\quad\quad\leq C_2\lambda^{l/2}.
\end{align*}
\end{proof}

The case of parabolic linear fractional automorphisms $\varphi$ is more delicate, because provided $f\neq 0$, the series $\sum_n \|C_\varphi^n(f)\|$ is never convergent
(see  \cite[Remark 3.7]{BAYGRIJFA}). We will need to estimate directly the norm of sums like $\sum_{n\in A}C_\varphi^n(f)$. Our next lemmas are inspired by \cite{BAYTCL} and \cite{Ta04}. 
\begin{lemma}\label{lem:sumpara1}
Let $F:\overline{\CC_0}\to\CC$ be such that $|F(it)|\leq C/(1+t^2)$ for all $t\in\RR$ and let $\tau\in\RR^*$. Then there exists $M>0$ such that, for all $t\in\RR$, 
$$\sum_{n\in\mathbb Z} |F(it+in\tau)|\leq M.$$
\end{lemma}
\begin{proof}
We just write 
\begin{align*}
\sum_{n\in\mathbb Z} |F(it+in\tau)|&\leq \sum_{n\in\mathbb Z}\frac C{1+(n\tau+t)^2}\\
&\leq \sup_{t'\in[0,\tau]}\sum_{n\in\mathbb Z}\frac C{1+(n\tau+t')^2}
\end{align*}
and the last sum defines, by uniform convergence, a continuous function on $[0,\tau]$. 
\end{proof}

We shall deduce the following substitute of Lemma \ref{lem:coh2hypsame}.

\begin{lemma}\label{lem:coh2parsame}
Let $F:\overline{\CC_0}\to\CC$ be such that $|F(it)|\leq C/(1+t^2)$ for all $t\in\RR$ and let $\tau\in\RR^*$. 
There exists a decreasing function $\veps:(0,+\infty)\to(0,+\infty)$ with $\lim_{+\infty}\veps=0$
 such that, for any $A\subset\NN$ finite, for any $a>0$, 
 $$\left\|\sum_{n\in A}F(\cdot+in\tau+ia)\right\|\leq\veps(a).$$
\end{lemma}
\begin{proof}
Let $M$ be given by Lemma \ref{lem:sumpara1}. For $t\geq -a/2$, we also have
 $$\sum_{n\in\NN} |F(it+in\tau+ia)|\leq\sum_{n\in\NN}\frac{C}{1+(n\tau+a/2)^2}=:\delta(a)$$
 where $\delta(a)$ decreases to zero as $a$ tends to $+\infty$. We then deduce that
 $$\left\|\sum_{n\in A}F(\cdot+in\tau+ia)\right\|^2\leq \int_{-\infty}^{-a/2} M^2\frac{dt}{\pi(1+t^2)}+\int_{-\infty}^{+\infty}\delta^2(a)\frac{dt}{\pi(1+t^2)}$$
 and the right hand side decreases clearly to $0$ as $a$ goes to $+\infty$.
\end{proof}
\begin{remark}\label{rem:coh2parsame}
 In particular, applying the lemma to $A'=A-\min(A)$ and $a=\min(A)\tau$, we get 
  $$\left\|\sum_{n\in A}F(\cdot+in\tau)\right\|\leq\veps(\min(A)).$$
\end{remark}


\subsection{Hyperbolic linear fractional maps with the same attractive fixed point}

We start the proof of Theorem \ref{thm:coh2} by studying the case where $\varphi_1$ and $\varphi_2$ are hyperbolic with the same attractive fixed point.

\begin{theorem}\label{thm:coh2hypsame}
Let $\varphi_1,\ \varphi_2\in \lfm$ be hyperbolic with the the same attractive fixed point $\alpha$. 
Assume that $\varphi_1'(\alpha)\neq \varphi_2'(\alpha)$. Then $C_{\varphi_1}$ and $C_{\varphi_2}$ are disjointly frequently hypercyclic on $H^2(\DD)$. 
\end{theorem}


\begin{proof}[Proof of Theorem \ref{thm:coh2hypsame}]
Without loss of generality, we may assume that $\alpha=1$, so that the half-plane models of $\varphi_1$ and $\varphi_2$ are respectively 
$$\psi_1(w)=\lambda(w-b_1)+b_1,\ \psi_2(w)=\mu(w-b_2)+b_2$$
with $\lambda,\mu>1$, $b_1,b_2\in(-1, 0]$. Since $\lambda\neq \mu$, we may assume $\mu=\lambda^r$ for some $r>1$. 
 Let $\beta_1=\mathcal C^{-1}(b_1)$, $\beta_2=\mathcal C^{-1}(b_2)$ and let 
$$\mathcal D=\left\{P\circ\mathcal C^{-1}:\ P\textrm{ holomorphic polynomial with }P(1)=P(\beta_1)=P(\beta_2)=0\right\}.$$
$\mathcal D$ is dense in $\mathcal H^2$. 
 Let $\omega>1$, $s>1$ with $rs<\omega$ and $r-s>0$. Define $(n_k)=\NN\cap \bigcup_{p\geq 0} (\omega^p,s\omega^p)$ and
 observe that $n_k\leq Ck$ for some $C> 0$.
 For $(F,G)\in\mathcal D^2$, we set 
 $$S_{n_k}(F,G)=F\left({\lambda^{-n_k}}(w-b_1)+b_1\right)+G\left(\mu^{-n_k}(w-b_2)+b_2\right).$$
 We aim to prove that the assumptions of Theorem \ref{thm:dfhcc} are satisfied.
 Lemma \ref{lem:coh2hypsame} implies immediately the unconditional convergence of $\sum_k S_{n_k}(F,G)$. Next let $k\geq 0$, $l\geq 0$ and let us observe that 
$$C_{\psi_1}^{n_k}S_{n_{k+l}}(F,G)=F\left(\lambda^{n_k-n_{k+l}}(w-b_1)+b_1\right)
+G\left(\lambda^{n_k-rn_{k+l}}(w-b_1)+\lambda^{-rn_{k+l}}(b_1-b_2)+b_2\right).$$
Since $n_k-rn_{k+l}\leq -l$, Lemma \ref{lem:coh2hypsame}, (b) implies that
$$\|C_{\psi_1}^{n_k}S_{n_{k+l}}(F,G)\|\lesssim \lambda^{-l/4}$$
which implies the unconditional convergence, uniformly in $k$, of $\sum_l C_{\psi_1}^{n_k}S_{n_{k+l}}(F,G)$.
This is the easiest part of (C2) where we do not need our particular choice of $(n_k)$. For the second half of the proof of (C2), we write
$$C_{\psi_2}^{n_k}S_{n_{k+l}}(F,G)=F\left(\lambda^{rn_k-n_{k+l}}(w-b_2)+\lambda^{-n_{k+l}}(b_2-b_1)+b_1\right)+G\left(\lambda^{r(n_k-n_{k+l})}(w-b_2)+b_2\right).$$
The difficulty here is to show the unconditional convergence, uniformly in $k$, of $$\sum_l F\left(\lambda^{rn_k-n_{k+l}}(w-b_2)+\lambda^{-n_{k+l}}(b_2-b_1)+b_1\right).$$
The idea behind the proof is that our choice of $(n_k)$ ensures that $rn_k-n_{k+l}$ is either sufficiently small or sufficiently big. Precisely, for $k\geq 0$, let 
 \begin{align*}
  p_0(k)&=\inf\left\{rn_k-n_{k+l}:\ l\geq 0,\ rn_k-n_{k+l}\geq 0\right\}\\
  p_1(k)&=\inf\left\{n_{k+l}-rn_k:\ l\geq 0,\ n_{k+l}-rn_k\geq 0\right\}\\
 \end{align*}
Lemma \ref{lem:coh2hypsame} yields that
$$\sum_{l}\left\|F\left(\lambda^{rn_k-n_{k+l}}(w-b_2)+\lambda^{-n_{k+l}}(b_2-b_1)+b_1\right)\right\|\leq \sum_{p\geq p_0(k)}\frac{C}{\lambda^{p/4}}+\sum_{p\geq p_1(k)}\frac{C}{\lambda^{p/4}}.$$
We claim that $p_0(k)$ and $p_1(k)$ go to $+\infty$ when $k$ goes to $+\infty$. Indeed, let $p$ be such that $n_k\in [\omega^p,s\omega^p]$
and assume that $rn_k-n_{k+l}\geq 0$. Let $q$ be such that $n_{k+l}\in [\omega^q,s\omega^q]$. Then $p\geq q$: indeed, $rs\omega^{p}-\omega^{p+1}<0$.
Thus, $rn_k-n_{k+l}\geq r\omega^p-s\omega^p=(r-s)\omega^p$ which is large if $k$ is large since $r-s>0$. Symmetrically, if $rn_{k}-n_{k+l}\leq 0$, 
then one shows that $q>p$ which implies $n_{k+l}-rn_k\geq (\omega-rs)\omega^p$. 

Therefore we have shown the existence of $(\veps_n)\in c_0(\mathbb N)$ such that, for all $k\geq 0$,
$$\sum_{l=0}^{+\infty}\|C_{\psi_2}^{n_k}S_{n_{k+l}}(F,G)\|<\veps_k.$$
This yields the unconditional convergence uniformly in $k$. 

The unconditional convergence uniformly in $k$ of $\sum_{l=0}^k C_{\psi_1}^{n_k}S_{n_{k-l}}(F,G)$ and $\sum_{l=0}^k C_{\psi_2}^{n_k}S_{n_{k-l}}(F,G)$
can be proved in a completely similar fashion (the former case being the most difficult one). Finally,
$$C_{\psi_1}^{n_k}S_{n_k}(F,G)=F(w)+G\left(\lambda^{(1-r)n_k}(w-b_1)+\lambda^{-rn_k}(b_1-b_2)+b_2\right)\to_{k\to+\infty}F$$
whereas similarly $C_{\psi_2}^{n_k}S_{n_k}(F,G)\to G$ (one may for instance again apply Lemma \ref{lem:coh2hypsame}).
\end{proof}


\subsection{Hyperbolic linear fractional maps with different attractive points}

We now investigate the case of two hyperbolic maps with different attractive fixed points. 

\begin{theorem}\label{thm:coh2hypdifferent}
Let $\varphi_1,\ \varphi_2\in \lfm$ be hyperbolic with different attractive fixed points.  Then $C_{\varphi_1}$ and $C_{\varphi_2}$ are disjointly frequently hypercyclic on $H^2(\DD)$. 
\end{theorem}

\begin{proof}[Proof of Theorem \ref{thm:coh2hypdifferent}]
Let $\alpha_1,\alpha_2$ (resp. $\beta_1,\beta_2$) be the attractive (resp. repulsive) fixed points of $\varphi_1$ and $\varphi_2$, $\lambda_1=\varphi_1'(\alpha_1)$, $\lambda_2=\varphi'_2(\alpha_2)$.
We define two dense sets $\mathcal D_1$ and $\mathcal D_2$ of $H^2(\DD)$ as follows: if $\alpha_1=\beta_2$, then $\mathcal D_2$ is the set of holomorphic polynomials vanishing at
$\alpha_2$ and $\beta_2$;
if $\alpha_1\neq\beta_2$, then $\mathcal D_2$ is the set of functions defined by Lemma \ref{lem:coh2hypdifferent3}
with $\varphi=\varphi_2$, $\gamma=\alpha_1$, $d=1$, adding to the sequence of zeros the points $\alpha_2$ and $\beta_2$.
$\mathcal D_1$ is defined accordingly, exchanging the roles played by $\varphi_1$ and $\varphi_2$. 

Let $(n_k)=(k+1)$ and let $(f,g)\in\mathcal D_1\times\mathcal D_2$. Define
$$S_{n_k}(f,g)=f\circ\varphi_1^{-n_k}+g\circ \varphi_2^{-n_k}.$$
 The unconditional convergence of $\sum_k S_{n_k}(f,g)$ follows easily from Lemma \ref{lem:coh2hypsame} or from Lemma \ref{lem:coh2hypdifferent5}. Moreover,
$$C_{\varphi_1}^{n_k}S_{n_{k+l}}(f,g)=f\circ \varphi_1^{-(n_{k+l}-n_k)}+g\circ \varphi_2^{-n_{k+l}}\circ\varphi_1^{n_k}.$$
Again, the unconditional convergence, uniformly in $k$, of $\sum_l C_{\varphi_1}^{n_k}S_{n_{k+l}}(f,g)$ is a consequence of Lemma \ref{lem:coh2hypdifferent5}. 
Let us now consider, for $l=0,\dots,k$, 
$$C_{\varphi_1}^{n_k}S_{n_{k-l}}(f,g)=f\circ \varphi_1^{n_k-(n_{k-l})}+g\circ\varphi_2^{-n_{k-l}}\circ\varphi_1^{n_k}.$$
The unconditional convergence, uniformly in $k$, of $\sum_{l=0}^k f\circ\varphi_1^{n_k-(n_{k-l})}$ is easy to prove using Lemma \ref{lem:coh2hypsame}. The real issue is the proof of the unconditional convergence, uniformly in $k$, of $\sum_{l=0}^k g\circ\varphi_2^{-(n_{k-l})}\circ\varphi_1^{n_k}$. 
The easy case is when $\alpha_1=\beta_2$, since then $\alpha_1$ is the attractive fixed point of both $\varphi_1$ and $\varphi_2^{-1}$. As above Lemma \ref{lem:coh2hypsame} yields 
$\|g\circ\varphi_2^{-q}\circ\varphi_1^p\|\leq C\lambda_2^{q/4}\lambda_1^{p/4}$,
which in turn yields the unconditional convergence uniformly in $k$ of the above series. 

Let us assume now that $\alpha_1\neq\beta_2$.
For $l\geq 1$ and $k\geq 1$, we evaluate $\|g\circ\varphi_2^{-l}\circ\varphi_1^{k}\|^2$ as follows:
we fix $\mu\in(\lambda_1,1)$ and consider $A_k=\{z\in\TT:\ \varphi_1^k(z)\notin D(\alpha_1,\mu^k)\}$. 
By Lemma \ref{lem:coh2hypdifferent4}, $A_k$ has measure less that $C(\lambda_1/\mu)^k$, with $\lambda_1/\mu \in(0,1)$ and we know that $g$ is bounded on the automorphism domain of $\varphi_2$. 
Therefore, 
$$\int_{A_k}|g\circ\varphi_2^{-l}\circ\varphi_1^k|^2 \leq C\left(\frac{\lambda_1}\mu\right)^k.$$

On the other hand, if $z\in\TT$ does not belong to $A_k$, then $\varphi_1^k(z)\in D(\alpha_1,\mu^k)$ so that, by Lemma \ref{lem:coh2hypdifferent1}, there 
exists $C\geq 1$ such that, for all $k,l\geq 1$, 
$\varphi_2^{-l}\circ\varphi_1^k(z)\in D(\gamma_l,C\mu^k \lambda_2^l)$ where we recall that $\gamma_l=\varphi_2^{-l}(\alpha_1)$. Therefore the properties of $g$ ensure that
$$|g\circ\varphi_2^{-l}\circ\varphi_1^k(z)|^2\leq C\mu^k,$$
an inequality which is true provided $\varphi_2$ is an automorphism or not.
Combining the two cases, we therefore have obtained 
\begin{equation}
 \|g\circ \varphi_2^{-l}\circ\varphi_1^k\|_2\leq C\rho^k \label{eq:thmcoh2hypdifferent}
\end{equation}
for some $\rho\in(0,1)$, with $C$ and $\rho$ which do not depend on $k$ and $l$. This yields the desired unconditional convergence of the involved series.
We conclude by symmetry of $\varphi_1$ and $\varphi_2$, and because (C4) follows also from \eqref{eq:thmcoh2hypdifferent}.
\end{proof}


\subsection{Parabolic automorphisms with the same attractive fixed point}
We turn to hyperbolic automorphisms, first with the same attractive fixed point.

\begin{theorem}\label{thm:coh2parsame}
 Let $\varphi_1,\varphi_2\in\aut$ be parabolic with the same attractive fixed point. Assume that $\varphi_1\neq\varphi_2$. Then $C_{\varphi_1}$ and $C_{\varphi_2}$ are 
 disjointly frequently hypercyclic on $H^2(\DD)$.
\end{theorem}

The proof will be inspired by that of Theorem \ref{thm:coh2hypsame} because we have to face the same difficulty: it is possible that $\varphi_1^k$ and $\varphi_2^{-l}$ compensate, namely that $\varphi_1^k\circ\varphi_2^{-l}$ is close to the identity map. We will have
to exclude this possibility by choosing for $(n_k)$ a subsequence of the whole sequence of integers. 

\begin{proof}[Proof of Theorem \ref{thm:coh2parsame}]
 Let us move to the half-plane and define $\psi_1(w)=w+i\tau_1$, $\psi_2(w)=w+i\tau_2$, $\tau_1\neq\tau_2\in\mathbb R^*$. 
 The proof depends on the sign of the product $\tau_1\tau_2$. If it is negative, namely if $\tau_1$ and $\tau_2$ have opposite signs,
 it is simpler. Let us start with this case, say $\tau_2<0<\tau_1$. We set 
 $$\mathcal D=\{P\circ\mathcal C^{-1}:\ P\textrm{ is a holomorphic polynomial},\ P(1)=P'(1)=0\}$$
 and pick $(F,G)$ in $\mathcal D^2$. Define 
 $$S_k(F,G)=F(w-ik\tau_1)+G(w-ik\tau_2).$$
 The unconditional convergence of $\sum_k S_k(F,G)$ comes from Lemma \ref{lem:coh2parsame} (and its consequence, Remark \ref{rem:coh2parsame}). For $k,l\geq 0$, 
 one has
 $$C_{\psi_1}^k S_{k+l}(F,G)(w)=F(w-il\tau_1)+G(w+ik\tau_1-i(k+l)\tau_2).$$
 The unconditional convergence uniformly in $k$ of $\sum_l G(w+ik\tau_1-i(k+l)\tau_2)$ follows again from Lemma \ref{lem:coh2parsame},
 writing
 $$k\tau_1-(k+l)\tau_2=k(\tau_1-\tau_2)-l\tau_2$$
 with $k(\tau_1-\tau_2)>0$. The remaining conditions of Theorem \ref{thm:dfhcc} are proved in the same way.
 
 Let us now assume that $\tau_1$ and $\tau_2$ have the same sign, for instance $\tau_2>\tau_1>0$ and write $\tau_2=r\tau_1$. To 
 simplify the notations, we simply denote $\tau$ for $\tau_1$. As in the proof of Theorem \ref{thm:coh2hypsame}, let $\omega,s>1$
 with $rs<\omega$ and $r-s>0$ and define $(n_k)=\mathbb N\cap \bigcup_{p\geq 0}(\omega^p, s\omega^p)$. For $(F,G)\in\mathcal D^2$, let
 $$S_{n_k}(F,G)(w)=F(w-in_k\tau)+G(w-in_kr\tau).$$
 The unconditional convergence of $\sum_k S_{n_k}(F,G)$ follows again from Lemma \ref{lem:coh2parsame}. For $k,l\geq 0$, we observe that
  $$C_{\psi_1}^{n_k} S_{n_{k+l}}(F,G)(w)=F(w+i(n_k-n_{k+l})\tau)+G(w+i(n_k-rn_{k+l})\tau).$$
  The unconditional convergence, uniformly in $k$, of $\sum_l C_{\psi_1}^{n_k} S_{n_{k+l}}(F,G)$ can be proved without
  any extra difficulties, because $r>1$ and $n_k-rn_{k+l}<c(k+l)$ for some $c<0$. On the contrary 
    $$C_{\psi_2}^{n_k} S_{n_{k+l}}(F,G)(w)=F(w+i(rn_k-n_{k+l})\tau)+G(w+ir(n_k-n_{k+l})\tau)$$
and this is our specific choice of $(n_k)$, which prevents $|rn_k-n_{k+l}|$ to be too small, which will ensure the unconditional convergence,
uniformly in $k$, of $\sum_l F(w+i(rn_k-n_{k+l})\tau)$. Indeed, let $A\subset\NN$ be finite and let 
 \begin{align*}
  p_0(k)&=\inf\left\{rn_k-n_{k+l}:\ l\geq 0,\ rn_k-n_{k+l}\geq 0\right\}\\
  p_1(k)&=\inf\left\{n_{k+l}-rn_k:\ l\geq 0,\ n_{k+l}-rn_k\geq 0\right\}.\\
 \end{align*}
We have already observed during the proof of Theorem \ref{thm:coh2hypsame} that $p_0(k)$ and $p_1(k)$ go to $+\infty$ with $k$.
We write $A=A_0\cup A_1$ with $A_0=\{l\in A:\ rn_k-n_{k+l}\geq 0\}$ and $A_1=\{l\in A:\ rn_k-n_{k+l}< 0\}$ so that 
\begin{align*}
 \left\|\sum_{l\in A}F(w+i(rn_k-n_{k+l})\tau)\right\|&\leq \left\|\sum_{l\in A_0}F(w+i(rn_k-n_{k+l})\tau)\right\|+\left\|\sum_{l\in A_1}F(w+i(rn_k-n_{k+l})\tau)\right\|\\
 &\leq \veps(p_0(k))+\veps(p_1(k))
\end{align*}
where $\veps$ is the function appearing in Lemma \ref{lem:coh2parsame}. We now conclude as above.
\end{proof}


\subsection{Parabolic automorphisms with different attractive fixed points}
The next result settles the case of parabolic automorphisms with different attractive fixed points.

\begin{theorem}\label{thm:coh2pardifferent}
 Let $\varphi_1,\varphi_2\in\aut\cap\para$ whose attractive fixed points are respectively $\alpha_1$ and $\alpha_2$. Assume that $\alpha_1\neq\alpha_2$. Then $C_{\varphi_1}$ and $C_{\varphi_2}$ are 
 disjointly frequently hypercyclic on $H^2(\DD)$.
\end{theorem}

\begin{proof}[Proof of Theorem \ref{thm:coh2pardifferent}]
Without loss of generality, we may assume $\alpha_1=-1$ and $\alpha_2=1$. Let $\mathcal D_1$ be the dense subset of $H^2(\DD)$ obtained by applying Lemma \ref{lem:coh2hypdifferent3} to $\varphi=\varphi_1$, 
$\gamma=\alpha_2$, $d=2$, assuming also that $\alpha_1$ is a zero of order $2$ of any function in $\mathcal D_1$. We define similarly $\mathcal D_2$, exchanging the roles played by $\varphi_1$ and $\varphi_2$. As usual, for $(f,g)\in\mathcal D_1\times\mathcal D_2$, define $S_{k}(f,g)=f\circ\varphi_1^{-k}+g\circ\varphi_2^{-k}$. The unconditional convergence of $\sum_k S_k(f,g)$ comes from Lemma \ref{lem:coh2parsame} (the key point here being that $f$ (resp. $g$) is bounded on $\overline{\DD}$ and $\alpha_1$ (resp. $\alpha_2$) is a zero of order $2$ of $f$ (resp. $g$)). For $k,l\geq 0$, 
$$C_{\varphi_1}^k S_{k+l}(f,g)=f\circ\varphi_1^{-l}+g\circ\varphi_2^{-(k+l)}\circ\varphi_1^k.$$
The first problem comes from the proof of the unconditional convergence, uniformly in $k$, of $\sum_l g\circ\varphi_2^{-(k+l)}\circ\varphi_1^k$. We move on the half-plane by considering $\psi_1(w)=\frac{w}{1+i\tau_1 w}$ (because $-1$ is the attractive fixed point of $\varphi_1$) , $\psi_2(w)=w+i\tau_2$, $\tau_2>0$, $G=g\circ\mathcal C^{-1}$
which satisfies $|G(it)|\leq C/(1+t^2)$ for some $C>0$. We can write
$$G\circ\psi_2^{-(k+l)}\circ\psi_1^{k}(it)=G\left(\frac{it}{1-k\tau_1 t}-(k+l)i\tau_2\right).$$
By Lemma \ref{lem:sumpara1}, there exists some constant $M>0$ such that, for all $k\geq 1$ and all $t\in\mathbb R$, 
\begin{equation}\label{eq:coh2pardifferent}
\sum_l |G\circ\psi_2^{-(k+l)}\circ\psi_1^{k}(it)|\leq M.
\end{equation}
Let $A\subset\NN$ be finite and let $2a=\tau_2\times \min(A)$. We set, for $k\geq 1$, 
$$I_k:=\left\{t\in\RR:\ \frac{t}{1-k\tau_1 t}\geq a\right\}.$$
Provided $t\notin I_k$, for all $l\in A$, 
\begin{align*}
\frac{t}{1-k\tau_1 t}-k\tau_2-l\tau_2&\leq a-l\tau_2\\
&\leq \frac{-l\tau_2}2
\end{align*}
which yields, as in the proof of Lemma \ref{lem:coh2parsame}, 
$$\sum_{l\in A}|G\circ \psi_2^{-(k+l)}\circ\psi_1^k(it)|\leq\veps(a)$$
where $\veps(a)$ does not depend neither on $t$ nor on $k$ (provided $t\notin I_k$) and $\veps(a)\to 0$ as $a\to+\infty$. 
Therefore, in view of \eqref{eq:coh2pardifferent} and Lemma \ref{lem:uniformunconditional} (a), one just need to prove that $\minv(I_k)\leq \delta(a)$, where $\delta(a)$ does not depend on $k$
and $\delta(a)\to 0$ as $a\to+\infty$. Assuming for instance $\tau_1>0$, it is easy to prove that
\begin{align*}
I_k&=\left[\frac{a}{1+ak\tau_1},\frac 1{k\tau_1}\right]
\end{align*}
which shows the claim since $\minv(I_k)\leq |I_k|\lesssim \frac 1a$ where the involved constant does not depend on $k$.

We now turn to 
$$C_{\varphi_1}^{k}S_{k-l}(f,g)=f\circ\varphi_1^{-l}+g\circ\varphi_2^{-(k-l)}\circ\varphi_1^{k}$$
and we prove the unconditional convergence, uniformly in $k$, of $\sum_{l=0}^k g\circ\varphi_2^{-(k-l)}\circ\varphi_1^k$. Let $C_1$ and $C_2>0$ be given by Lemma \ref{lem:coh2pardifferent} for $\varphi_1$. For $k\geq 0$, let us set 
$$E_k=\left\{z\in\TT:\ \varphi_1^k(z)\notin D(\alpha_1,C_1k^{-1})\right\}.$$
The measure of $E_k$ is less than $C_2k^{-1}$ and, by Lemma \ref{lem:sumpara1}, 
$\sum_{l=0}^{+\infty}|g\circ\varphi_2^{-(k-l)}\circ\varphi_1^k(z)|\leq M$ for some $M>0$, independent of $z\in\TT$ and $k\in\NN$. Therefore, for any $A\subset\NN$ finite, 
$$\left(\int_{E_k}\left|\sum_{l\in A\cap [0,k]}g\circ\varphi_2^{-(k-l)}\circ\varphi_1^k\right|^2 dm\right)^{1/2}\lesssim k^{-1/2}.$$
On the other hand, let $z\notin E_k$. We know that, for $k$ large enough, $\alpha_2\notin D(\alpha_1,C_1k^{-1})$. One may use Lemma \ref{lem:coh2hypdifferent1} to get $\varphi_2^{-(k-l)}\circ\varphi_1^k(z)\in D(\gamma_{k-l},C_3 k^{-1}(k-l)^{-2})$
for some $C_3>0$ where $\gamma_{k-l}=\varphi_2^{-(k-l)}(\alpha_1)$. Using Lemma \ref{lem:coh2hypdifferent3}, we obtain, for all $z\notin E_k$, 
$$|g\circ\varphi_2^{-(k-l)}\circ\varphi_1^{k}(z)|\lesssim (k-l)^4\times\frac{1}{k^2(k-l)^4}=\frac 1{k^2}.$$
Hence,
\begin{align*}
\left(\int_{E_k^c}\left|\sum_{l\in A\cap[0,k]}g\circ\varphi_2^{-(k-l)}\circ\varphi_1^{k}\right|^2dm\right)^{1/2}&\leq \sum_{l\in A\cap[0,k]}\left(\int_{E_k^c}\left|g\circ\varphi_2^{-(k-l)}\circ\varphi_1^{k}\right|^2dm\right)^{1/2}\\
&\lesssim \frac{1}k.
\end{align*}
This achieves the proof of the unconditional convergence, uniformly in $k$, of $\sum_l g\circ\varphi_2^{-(k-l)}\circ\varphi_1^k$ and that of Theorem \ref{thm:coh2pardifferent}.
\end{proof}


\subsection{Parabolic automorphism and hyperbolic automorphism with different attractive fixed points}

In this subsection, we turn to the case where $\varphi_1$ is parabolic and $\varphi_2$ is hyperbolic. We begin with the case where they have different attractive fixed points. 

\begin{theorem}\label{thm:coh2parhypdifferent}
Let $\varphi_1\in \para\cap\aut$, $\varphi_2\in\hyp$ with different attractive fixed points. Then $C_{\varphi_1}$ and $C_{\varphi_2}$ are disjointly frequently hypercyclic.
\end{theorem} 
\begin{proof}
Let $\alpha_i$ be the attractive fixed point of $\varphi_i$, $i=1,2$ and let $\beta_2$ be the repulsive fixed point of $\varphi_2$.
We may assume $\alpha_1=1$ and $\alpha_2=-1$. We also set $\lambda=\varphi_2'(\alpha_2)$. The proof will differ at some point provided $\beta_2=\alpha_1$ or not. 
We define  $\mathcal D_1$ as the set given by applying Lemma \ref{lem:coh2hypdifferent3} for $\varphi_1$, $\gamma=\alpha_2$, $d=2$, adding that $\alpha_1$ is a zero of multiplicity $2$ of any $f\in\mathcal D_1$. When $\beta_2\neq\alpha_1$, we define $\mathcal D_2$ as the set given by Lemma \ref{lem:coh2hypdifferent3} for $\varphi_2$,
$\gamma=\alpha_1$, $d=2$, adding $\alpha_2$ and $\beta_2$ at the sequence of zeros.
Otherwise, we consider for $\mathcal D_2$ the set of polynomials with a zero of order at least 2 at $\alpha_2$ and $\beta_2$. 

For $(f,g)\in\mathcal D_1\times\mathcal D_2$, we consider $S_k(f,g)=f\circ\varphi_1^{-k}+g\circ\varphi_2^{-k}$; the unconditional convergence of $\sum_k S_k(f,g)$ is now standard in this paper. Let us turn to $\sum_l C_{\varphi_1}^k S_{k+l}(f,g)$. The difficult point is the unconditional convergence uniformly in $k$ of $\sum_l g\circ\varphi_2^{-(k+l)}\circ\varphi_1^k$. It follows from an application of Lemma \ref{lem:coh2hypdifferent5}.
We now investigate $\sum_l C_{\varphi_2}^k S_{k+l}(f,g)$ - observe that now there is no symmetry between $\varphi_1$ and $\varphi_2$ and we have to investigate each case. We are reduced to the study of $\sum_l f\circ\varphi_1^{-(k+l)}\circ\varphi_2^k$. We move on the half-plane and study $\sum_l F\circ\psi_1^{-(k+l)}\circ\psi_2^k$, where 
$$F\circ\psi_1^{-(k+l)}\circ\psi_2^k(it)=F\big(\psi_2^k(it)-(k+l)i\tau\big),$$
assuming $\psi_1(w)=w+i\tau$ for some $\tau>0$. We argue as in the proof of Theorem \ref{thm:coh2pardifferent} setting, for $a>0$, 
$$I_k(a):=\{t\in\RR:\ \Im m(\psi_2^k(it))\geq a\}\subset\{t\in\RR:\ |\psi_2^k(it)|\geq a\}.$$
We just need to prove that $\minv(I_k(a))\leq\delta(a)$ with $\delta(a)\to 0$ as $a\to+\infty$.
Since $\psi_2$ is the linear fractional map fixing $0$ and $b_2=\mathcal C(\beta_2)$ with $\psi_2'(0)=\lambda$, it can be written
$$\psi_2(w)=\frac{\lambda b_2 w}{(\lambda-1)w+b_2}$$
(if $b_2=\infty$, namely $\beta_1=\alpha_1=1$, $\psi_2$ is simply the dilation $\psi_2(w)=\lambda w$ and the required inequality on $\minv(I_k(a))$ is immediate in this case).
Therefore,  
$$ t\in I_k(a)\implies  \frac{\lambda^{2k} |b_2|^2 t^2}{\left((\lambda^k-1)t+\Im m(b_2)\right)^2+\left(\Re e(b_2)\right)^2}\geq a^2.$$
If $\Re e(b_2)\neq 0$ (namely, if $\varphi_2$ is not an automorphism), then we deduce easily that $|t|\gtrsim a$ (where the involved constant
does not depend on $k$) which shows what we need on $\minv(I_k(a))$. Otherwise we let
\begin{align*}
 I_k^1(a)&=\{t\in \RR:\ |t|\geq \sqrt a/|b_2|\}\\
 I_k^2(a)&=\{t\in\RR:\ |(\lambda^k-1)t+\Im m(b_2)|^2\leq a^{-1}\}
\end{align*}
and observe that $I_k(a)\subset I_k^1(a)\cup I_k^2(a)$. Again, $\minv(I_k^1(a))\leq\delta(a)$. Regarding $I_k^2(a)$, 
let us set $t_k=\Im m(b_2)/(1-\lambda^k)$. Then
$$I_k^2(a)= \left[t_k-\frac{a^{-1/2}}{1-\lambda^k},t_k+\frac{a^{1/2}}{1-\lambda^k}\right]\subset \left[t_k-\frac{a^{-1/2}}{1-\lambda},t_k+\frac{a^{-1/2}}{1-\lambda}\right]$$
so that
$$\minv(I_k^2(a))\leq meas\left(\left[t_k-\frac{a^{-1/2}}{1-\lambda},t_k+\frac{a^{-1/2}}{1-\lambda}\right]\right)\leq \frac{2 a^{-1/2}}{1-\lambda}.$$
Hence, the estimation needed on $\minv(I_k(a))$ is also true in this case.


The next step is the investigation of $\sum_{l=0}^k C_{\varphi_1}^k S_{k-l}(f,g)$ which essentially reduces to the study of $\sum_{l=0}^k g\circ\varphi_2^{-(k-l)}\circ \varphi_1^k$. 
Let $C_1,C_2>0$ be given by Lemma \ref{lem:coh2pardifferent} for $\varphi_1$. For $k\geq 0$, let us set 
$$E_k:=\left\{z\in\TT:\ \varphi_1^k(z)\notin D(\alpha_1,C_1 k^{-1})\right\}$$
which satisfies $m(E_k)\leq C_2k^{-1}$. Since $\sum_m |g\circ\varphi_2^{-m}|$ is bounded 
on $\TT$, we easily deduce the existence of $M>0$, independent of $k$, such that, 
for any $A\subset\NN$ finite, 
$$\left(\int_{E_k}\left|\sum_{l\in A\cap [0,k]} g\circ\varphi_2^{-(k-l)}\circ\varphi_1^k\right|^2dm\right)^{1/2}\leq\frac M{\sqrt k}.$$
To evaluate the integral over $E_k^c$, we first assume that $\alpha_1\neq\beta_2$, so that,  for $k$ large enough, $\alpha_2,\beta_2\notin D(\alpha_1,C_1k^{-1})$.
One may use Lemma \ref{lem:coh2hypdifferent1} to get $\varphi_2^{-(k-l)}\circ\varphi_1^k(z)\in D(\gamma_{k-l},C_3k^{-1}\lambda^{k-l})$ where $\gamma_{k-l}=\varphi_2^{-(k-l)}(\alpha_2)$
provided $z\notin E_k$. This yields 
$$|g\circ\varphi_2^{-(k-l)}\circ\varphi_1^k(z)|\lesssim
\left\{
\begin{array}{ll}
 \displaystyle(k-l)^4\frac{\lambda^{2(k-l)}}{k^2}&\textrm{if $\varphi_2$ is an automorphism}\\[0.2cm]
 \displaystyle\lambda^{2(l-k)}\frac{\lambda^{2(k-l)}}{k^2}&\textrm{otherwise.} 
\end{array}
\right.
$$
In both cases, this implies 
\begin{equation}\label{eq:coh2parhypdifferent}
|g\circ\varphi_2^{-(k-l)}\circ\varphi_1^k(z)|\lesssim \frac1{k^2}.
\end{equation}
Hence,
\begin{align*}
\left(\int_{E_k^c} \left| \sum_{l\in A\cap[0,k]} g\circ\varphi_2^{-(k-l)}\circ\varphi_1^k\right|^2 dm\right)^{1/2}&\lesssim \sum_{l\in A\cap[0,k]}  \left(\int_{E_k^c} \frac{1}{k^4}dm\right)^{1/2}\\
&\lesssim \frac{1}{k}.
\end{align*}
When $\alpha_1=\beta_2$,
the proof is simpler since for any $z\notin E_k$, $\varphi_2^{-(k-l)}\circ\varphi_1^k(z)\in D(\beta_2,Ck^{-1}\lambda^{k-l})$.
Therefore \eqref{eq:coh2parhypdifferent} remains true since $g$ has a zero of order $2$ at $\beta_2$.

To conclude, it remains to handle $\sum_{l=0}^k C_{\varphi_2}^kS_{k-l}(f,g)$, namely
$\sum_{l=0}^k f\circ\varphi_1^{-(k-l)}\circ\varphi_2^k$. We set 
$$E'_k:=\left\{z\in\TT:\ \varphi_2^k(z)\notin D(\alpha_2,k^{-1})\right\}.$$
Then, by Lemma \ref{lem:coh2hypdifferent4}, $\minv(E'_k)\leq C_1\lambda^k k^{-1}$ tends to zero as $k$ tends to $+\infty$.
On the other hand, if $z\notin E'_k$, then $\varphi_1^{-(k-l)}\circ\varphi_2^k(z)\in D(\gamma'_{k-l},C_2 k^{-1}(k-l)^{-2})$,
where $\gamma'_{k-l}=\varphi_1^{-(k-l)}(\alpha_2)$. We conclude exactly as in the proof of Theorem \ref{thm:coh2pardifferent}. 
\end{proof}


\subsection{Parabolic automorphism and hyperbolic automorphism with the same attractive fixed point}

We shall prove a result which is stronger than required. 

\begin{theorem}
Let $\varphi_1\in\para\cap\aut$, $\varphi_2\in\hyp\cap \aut$ with the same attractive fixed point. Then any frequently hypercyclic vector for $C_{\varphi_1}$ is not a hypercyclic 
vector for $C_{\varphi_2}$. 
\end{theorem}
\begin{proof}
We may assume that the common attractive fixed point is $1$ and that the repulsive fixed point of $\varphi_2$ is $-1$.
We move on the half-plane and consider $\psi_1(w)=w+i\tau$, $\tau>0$, $\psi_2(w)=\lambda w$, $\lambda>1$. 
We argue by contradiction and assume that $F\in FHC(C_{\psi_1})\cap HC(C_{\psi_2})$. 
There exist $C>0$ and a sequence $(n_k)$ such that, for all $k\geq 1$, $n_k\leq Ck$ and $C_{\psi_1}^{n_k}F$ is so close to $1$ that 
$$\int_0^\tau \big |C_{\psi_1}^{n_k}(F)(it)|^2dt\geq \frac12\iff \int_{n_k\tau}^{n_k\tau+\tau}|F(it)|^2dt\geq\frac1 2.$$
On the other hand, let $\veps\in(0,1/4C\tau)$ and $p\in\mathbb N$ as large as we want such that $C_{\psi_2}^p(F)$ is so close to $0$ that
$$\int_0^1 |F\circ \psi_2^p(it)|^2dt<\veps.$$
This yields 
$$\int_0^{\lambda^p} |F(it)|^2dt\leq\lambda^p \veps.$$
Now, provided $p$ is large enough, there are at least $\lambda^p/2C\tau$ disjoint intervals $(n_k\tau,n_k\tau+\tau)$ contained in $[0,\lambda^p]$. Therefore, 
$$\lambda^p \veps \geq \int_0^{\lambda^p }|F(it)|^2 dt\geq \frac{\lambda^p}{2C\tau}\times \frac{1}2.$$
This contradicts the choice of $\veps$. 
\end{proof}

\begin{remark}
 The same proof shows that any upper frequently hypercyclic vector for $C_{\varphi_1}$ is not a hypercyclic vector for $C_{\varphi_2}$.
\end{remark}


\subsection{Parabolic automorphism and hyperbolic nonautomorphism with the same attractive fixed point}

We finish the proof of Theorem \ref{thm:coh2}  with the last case.

\begin{theorem}
Let $\varphi_1\in\para\cap\aut$, $\varphi_2\in\hyp\backslash \aut$ with the same attractive fixed point. 
Then $C_{\varphi_1}$ and $C_{\varphi_2}$ are disjointly frequently hypercyclic.
\end{theorem}
\begin{proof}
 We assume that $\alpha_1=\alpha_2=1$ and write $\psi_1(w)=w+i\tau$, $\tau>0$, $\psi_2(w)=\lambda(w-b_2)+b_2$, $b_2\in (-\infty,0)$ and $\lambda>1$.
 We define $\mathcal D_1$ and $\mathcal D_2$ as follows: $\mathcal D_1$ is the set of holomorphic polynomials with a zero of order $2$ at $1$, 
 $\mathcal D_2$ is the set given by Lemma \ref{lem:coh2hypdifferent3} for $\varphi_2$, $\gamma=\alpha_1$, $d=2$, adding $\alpha_2$ and $\beta_2$
 at the sequence of zeros.
 
As usual, for $(f,g)\in\mathcal D_1\times\mathcal D_2$, let $S_k(f,g)=f\circ\varphi_1^{-k}+g\circ\varphi_2^{-k}$. We first study the
series $\sum_l f\circ\varphi_1^{-(k+l)}\circ\varphi_2^k$. Let $t\in\mathbb R$. Then
\begin{align*}
 \sum_{l=0}^{+\infty}|F\circ\psi_1^{-(k+l)}\circ\psi_2^k(it)|&=\sum_{l=0}^{+\infty}\big|F\big(\lambda^k(it-b_2)+b_2-(k+l)i\tau\big)\big|\\
 &\leq \sum_{l=0}^{+\infty} \frac{C}{(-\lambda^k b_2+b_2+1)^2+(\lambda^k-k\tau-l\tau)^2}\\
 &\leq 2C\sum_{l\in\mathbb Z} \frac1{(-\lambda^k b_2+b_2+1)^2+l^2\tau^2}\\
 &\leq \veps_k
\end{align*}
with $\veps_k\to 0$ as $k\to+\infty$. In particular, for all $A\subset\NN$ finite, 
$$\left\|\sum_{l\in A}F\circ\psi_1^{-(k+l)}\circ\psi_2^k\right\|\leq\veps_k$$
which shows the unconditional convergence, uniformly in $k$, of $\sum_l f\circ\varphi_1^{-(k+l)}\circ\varphi_2^k$. 
The corresponding property for the series $\sum_{l=0}^k f\circ\varphi_1^{-(k-l)}\circ\varphi_2^k$ follows exactly the same lines.
Let us turn to $\sum_l g\circ\varphi_2^{-(k+l)}\circ\varphi_1^k$ and $\sum_{l=0}^k g\circ\varphi_2^{-(k-l)}\circ\varphi_1^k$.
But there is nothing new: we just follow line by line the corresponding estimates of Theorem \ref{thm:coh2parhypdifferent},
in the case $\beta_2\neq \alpha_1$.
\end{proof}


\section{Composition operators on $H(\DD)$}

\label{sec:cohol}

\subsection{Preliminaries}

Let $\varphi$ be a self-map of $\DD$. It is well-known that $C_\varphi$ is hypercyclic on $H(\DD)$ if and only if $\varphi$ is univalent
on $\DD$ and has no fixed point in $\DD$. In that case there exists a unique point $\alpha\in\TT$, the Denjoy-Wolff point of $\varphi$,
such that $\varphi^n\to\alpha$ uniformly on all compact subsets of $\DD$. Moreover $\varphi$ admits an angular derivative
$\varphi'(\alpha)$ at $\alpha$ with $\varphi'(\alpha)\in (0,1)$. $\varphi$ is called {\bf parabolic} if $\varphi'(\alpha)=1$
and {\bf hyperbolic} if $\varphi'(\alpha)\in (0,1)$. 
Our proof of Theorem \ref{thm:cohol} will need to know precisely
how $(\varphi^n(0))$ converges to $\alpha$. For this we will require some regularity of $\varphi$ at $\alpha$ in order to apply the results of \cite{BoSh97}.
\begin{definition}
 Let $\varphi$ be a parabolic univalent self-map of $\DD$ with Denjoy-Wolff point $\alpha$. We say that $\varphi$
 is {\bf regular} if $\varphi$ is continuous on $\overline{\DD}$, $\varphi''(\alpha)\neq 0$ and there exists $\veps>0$ such that
 $$\varphi(z)=\alpha+(z-\alpha)+\frac{\varphi''(\alpha)}{2}(z-\alpha)^2+\frac{\varphi^{(3)}(\alpha)}{3!}(z-\alpha)^3+\gamma(z)$$
 where $\gamma(z)=o(|z-\alpha|^{3+\veps})$ as $z\to \alpha$ in $\DD$.
\end{definition}
We then know that such a map $\varphi$ satisfies $\Re e(\varphi''(\alpha))\geq 0$ and that
\begin{itemize}
 \item if $ \Re e(\varphi''(\alpha))>0$, then the orbits $(\varphi^n(z))$ converge nontangentially to $\alpha$;
 \item if $\varphi''(\alpha)$ is pure imaginary, then the orbits $(\varphi^n(z))$ converge tangentially to $\alpha$.
\end{itemize}
\begin{definition}
  Let $\varphi$ be a hyperbolic univalent self-map of $\DD$ with Denjoy-Wolff point $\alpha$. We say that $\varphi$
 is {\bf regular} if $\varphi$ is continuous on $\overline{\DD}$ and there exists $\veps>0$ such that
 $$\varphi(z)=\alpha+\varphi'(\alpha)(z-\alpha)+\gamma(z)$$
 where $\gamma(z)=o(|z-\alpha|^{1+\veps})$ as $z\to \alpha$ in $\DD$.
\end{definition}
For a hyperbolic map, the orbits $(\varphi^n(z))$ always converge nontangentially to $\alpha$ without assuming any regularity at $\alpha$.

\smallskip

As for $H^2(\DD)$, we will often move on the right half-plane. 
If $\varphi$ is a self-map of $\DD$, $\psi$ will always mean $\psi=\mathcal C\circ\varphi\circ\mathcal C^{-1}$.
$\rho_{\mathbb D}$ will denote the hyperbolic distance on the disc, and $\rho_{\CC_0}$ the hyperbolic distance on the right half-plane $\CC_0$. We recall that 
$$\rho_{\CC_0}(w_1,w_2)=\log\frac{1+p_{\CC_0}(w_1,w_2)}{1-p_{\CC_0}(w_1,w_2)}$$
where $p_{\CC_0}$ is the pseudo-hyperbolic distance on $\CC_0$ given by
$$p_{\CC_0}(w_1,w_2)=\frac{|w_1-w_2|}{|w_1+\overline{ w_2}|}.$$

The following lemma is reminiscent from \cite[Theorem 2.6]{BCJP18}.
\begin{lemma}\label{lem:criterioncomposition}
Let $\varphi_1$, $\varphi_2$ be holomorphic and univalent self-maps of $\DD$ without fixed point in $\DD$.
Let $(K_k)$ be an increasing sequence of compact discs contained in $\mathbb D$ such that $\mathbb D=\bigcup_{k\geq 1}K_k$.
Assume that there exists a family $(A_k)_{k\in\NN}$ of pairwise disjoint subsets of $\NN$ such that 
\begin{enumerate}[(a)]
\item each $A_k$ has positive lower density;
\item the sets $\varphi_i^{n}(K_k)$, $k\in\mathbb N$, $n\in A_k$, $i=1,2$, are pairwise disjoint.
\end{enumerate}
Then $C_{\varphi_1}$ and $C_{\varphi_2}$ are disjointly frequently hypercyclic on $H(\DD)$. 
\end{lemma}
\begin{proof}
Since the proof follows rather closely that of \cite[Theorem 3.3]{BCJP18}, we will only sketch it. Let $(P_l,Q_l)$ be a sequence of couples of polynomials
which is dense in $H(\DD)\times H(\DD)$. Let us split each $A_k$, $k\in\NN$, into infinitely many disjoint subsets $A_k(l)$, $l\in\NN$, each of them 
being of positive lower density. We set 
$$F=\bigcup_{i=1}^2 \bigcup_{k\geq 1}\bigcup_{n\in A_k}\varphi_i^n(K_k).$$
Our assumption implies, in the terminology of \cite{Gai87}, that $F$ is a Carleman set of $\DD$; namely for all $f$ continuous on $F$
and analytic in $\mathring{F}$, there exists $g\in H(\DD)$ such that 
$$|f(z)-g(z)|\leq \veps(|z|)$$
where $\veps(r)\to 0$ as $r\to 1$. We apply this with 
$$f=\left\{
\begin{array}{ll}
 P_l\circ \varphi_1^{-n}&\textrm{on }\varphi_1^n(K_k),\ n\in A_k(l)\\
  Q_l\circ \varphi_2^{-n}&\textrm{on }\varphi_2^n(K_k),\ n\in A_k(l).
\end{array}\right.
$$
It is then easy to check that $f$ is a disjointly frequently hypercyclic vector for $(C_{\varphi_1},C_{\varphi_2})$.
\end{proof}

The next lemma will help us to produce sets $(A_k)$ such that assumption (b) of Lemma \ref{lem:criterioncomposition} is satisfied.

\begin{lemma}\label{lem:runaway}
Let $\varphi$ be a univalent self-map of $\DD$ with attractive fixed point $\alpha\in \TT$. Then for any compact subset $K$ of $\DD$, there exists $N\in\NN$ such that, for all $n,m$ with $|n-m|\geq N$, $\varphi^n(K)\cap\varphi^m(K)=\varnothing$.
\end{lemma}
\begin{proof}
By injectivity of $\varphi$, it suffices to prove that $\varphi^n(K)\cap K=\varnothing$ for all sufficiently large $n$. This is a consequence of the Denjoy-Wolff theorem, since $\varphi^n(z)$ converges uniformly to $\alpha$ on $K$.
\end{proof}

We are now ready for the proof of Theorem \ref{thm:cohol}, which will be divided into several cases. 
 Part of the proof will be done in a more general context (without assuming any regularity at the boundary of the maps).

\subsection{Well-separated convergence to the boundary}

We first show a very general result stating that, providing the sequences $(\varphi_1^n(0))$ and $(\varphi_2^n(0))$ are far away for the hyperbolic distance, 
then $C_{\varphi_1}$ and $C_{\varphi_2}$ are disjointly frequently hypercyclic.

\begin{theorem}\label{thm:wellseparated}
Let $\varphi_1$, $\varphi_2$ be holomorphic and univalent self-maps of $\DD$ without fixed point in $\DD$. Assume that, for all $A>0$, 
there exists $N\in\NN$ such that, for all $m,n\geq N$, $\rho_{\DD}(\varphi_1^m(0),\varphi_2^n(0))\geq A$. Then $C_{\varphi_1}$ and $C_{\varphi_2}$ are disjointly frequently hypercyclic on $H(\DD)$.
\end{theorem}
\begin{proof}
We shall apply Lemma \ref{lem:criterioncomposition}. Let $(R_k)$ be a sequence of positive real numbers increasing to $+\infty$ and let $D_k=\overline{D_{\rho_{\DD}}}(0,R_k)$ 
(the closed disc for the hyperbolic distance). 
By Lemma \ref{lem:runaway}, for any $k\geq 1$, there exists an integer $N_k^{(1)}$ such that, for all $n,m$ with $|n-m|\geq N_k^{(1)}$, for $i=1,2$, 
\begin{equation}\label{eq:wellseparated1}
\varphi_i^n(D_k)\cap\varphi_i^m(D_k)=\varnothing.
\end{equation}
Let us also consider $N_k^{(2)}$ such that 
\begin{equation}\label{eq:wellseparated2}
n,m\geq N_k^{(2)}\implies \rho_{\DD}(\varphi_1^n(0),\varphi_2^m(0))>2R_k.
\end{equation}
Since the sequences $(\varphi^n_1(0))$ and $(\varphi^n_2(0))$ go to the unit circle, we may also find $N_k^{(3)}$ such that
\begin{equation}\label{eq:wellseparated3}
m<N_k^{(2)}\textrm{ and }n\geq N_k^{(3)}\implies \rho_{\DD}(\varphi_1^n(0),\varphi_2^m(0))>2R_k
\end{equation}
\begin{equation}\label{eq:wellseparated4}
n<N_k^{(2)}\textrm{ and }m\geq N_k^{(3)}\implies \rho_{\DD}(\varphi_1^n(0),\varphi_2^m(0))>2R_k.
\end{equation}
We define $N_k=\max(N_k^{(1)},N_k^{(2)},N_k^{(3)})$. Let $(A_k)$ be a family of pairwise subsets of $\NN$ such that, for every $k\geq 1$, $\min(A_k)\geq N_k$ and,
for every $n\in A_k$ and every $m\in A_l$ with $n\neq m$, $|n-m|\geq N_k+N_l$. We claim that the sets $\varphi_i^{n}(D_k)$, $k\in\mathbb N$, $n\in A_k$, $i=1,2$, are pairwise disjoint. 
Indeed, let us consider two different couples $(i,n)$ and $(j,m)$ with $n\in A_k$ and $m\in A_l$. If $i=j$, then \eqref{eq:wellseparated1} implies that, since $|n-m|\geq N^{(1)}_{\max(k,l)}$, 
$$\varphi_i^{n}(D_{\max(k,l)})\cap\varphi_i^m(D_{\max(k,l)})=\varnothing.$$
If $i\neq j$, then we may assume $k\geq l$. Following the position of $m$ with respect to $N_k^{(2)}$, \eqref{eq:wellseparated2} or \eqref{eq:wellseparated3} ensure that
$$D_{\rho_\DD}(\varphi_1^n(0),R_k)\cap D_{\rho_\DD}(\varphi_2^m(0),R_l))=\varnothing.$$
Since any analytic self-map is a contraction with respect to the hyperbolic distance, this gives $\varphi_1^n(D_k)\cap\varphi_2^m(D_l)=\varnothing$. We conclude by applying Lemma \ref{lem:criterioncomposition}.
\end{proof}
	
We now give two interesting corollaries.

\begin{corollary}\label{cor:differentattractive}
Let $\varphi_1$, $\varphi_2$ be univalent self-maps of $\DD$ with respective Denjoy-Wolff points $\alpha_1$, $\alpha_2\in \TT$. Assume that $\alpha_1\neq\alpha_2$. 
Then $C_{\varphi_1}$ and $C_{\varphi_2}$ are disjointly frequently hypercyclic on $H(\DD)$.
\end{corollary}

\begin{corollary}\label{cor:tangentialandnontangential}
Let $\varphi_1$, $\varphi_2$ be univalent self-maps of $\DD$ with the same Denjoy-Wolff point $\alpha\in \TT$. Assume that $(\varphi_1^n(0))$ converges nontangentially to $\alpha$ and 
that $(\varphi_2^n(0))$ converges tangentially to $\alpha$. Then $C_{\varphi_1}$ and $C_{\varphi_2}$ are disjointly frequently hypercyclic on $H(\DD)$.
\end{corollary}

In particular, if $\varphi_1$ is a hyperbolic univalent self-map of $\DD$ and if $\varphi_2$ is a univalent regular parabolic self-map of $\DD$,
with the same Denjoy-Wolff point $\alpha$ and $\varphi_2''(\alpha)\in i \mathbb R$, then we may apply Corollary \ref{cor:tangentialandnontangential}. 
This is also the case if $\varphi_1$ and $\varphi_2$ are both univalent regular parabolic self-maps of $\DD$, with $\Re e(\varphi_1''(\alpha))>0$ and $\varphi_2''(\alpha)\in i\mathbb R$
(although this last case will be also covered later in greater generality).
\begin{proof}
The computations are simplified if we work on the half-plane $\mathbb C_0$. 
We let $\psi_2^n(0)=:x_n+iy_n$; tangential convergence implies that $|y_n|\to+\infty$ and $|y_n|/|x_n|\to+\infty$. Let $I=[1,+\infty)$.
We first show that $\rho_{\CC_0}(\psi_2^n(0),I)\xrightarrow{n\to+\infty}+\infty$. It suffices to show that,
for all $\delta\geq 0$, there exists $n_0\in\NN$ such that, for all $n\geq n_0$, for all $x\geq 1$,
$$\left|\frac{x_n+iy_n-x}{x_n+iy_n+x}\right|^2\geq 1-\delta.$$
Let $\veps\in(0,1)$ and $n_1\in\NN$ be such that $n\geq n_1$ implies $x_n\leq\veps^2 |y_n|$. Let $x\geq 1$ and $n\geq n_1$. Assume first that $x\geq \veps |y_n|$. Then
$$|x-x_n|=x-x_n\geq (1-\veps)x\textrm{ and }|x_n+x|\leq (1+\veps)x$$
so that
$$
\left|\frac{x_n+iy_n-x}{x_n+iy_n+x}\right|^2\geq\frac{(1-\veps)^2 x^2+|y_n|^2}{(1+\veps)^2 x^2+|y_n|^2}\geq 1-\delta$$
provided $n$, hence  $|y_n|$, is large enough. Otherwise, $x\leq\veps|y_n|$ and 
$$\left|\frac{x_n+iy_n-x}{x_n+iy_n+x}\right|^2\geq\frac{|y_n|^2}{|y_n|^2+4\veps^2 |y_n|^2}\geq 1-\delta$$
provided $\veps>0$ is small enough.
Let us now conclude that the assumptions of Theorem \ref{thm:wellseparated} are satisfied. We write $\psi_1^m(0)=:u_m+iv_m$ and observe that 
$$\left|\frac{u_m+iv_m-u_m}{u_m+iv_m+u_m}\right|^2=\frac{|v_m|^2}{4|u_m|^2+|v_m|^2}$$
is bounded away from 1 since, by nontangential convergence, $|v_m|\leq M|u_m|$ for some $M>0$. In particular, $\rho_{\CC_0}(\psi_1^m(0),I)\leq C$ for some $C>0$. Using the triangle inequality, we get that
$\lim_{n\to+\infty}\inf_{m\in\NN}\rho_{\CC^+}\big(\psi_1^m(0),\psi_2^n(0)\big)=+\infty$. We conclude from Theorem \ref{thm:wellseparated} that $C_{\varphi_1}$ and $C_{\varphi_2}$ are disjointly frequently hypercyclic.
\end{proof}

\subsection{Two parabolic maps with the same boundary fixed point}

We investigate the case where $\varphi_1$ and $\varphi_2$ are regular parabolic self-maps of $\DD$ sharing the same attractive fixed point on the circle, say $+1$. 
We shall need the following result on the behaviour of the half-plane model of such a map.

\begin{lemma}\label{lem:orbitsparabolic}
 Let $\varphi$ be a univalent regular parabolic self-map of $\DD$ with $+1$ as Denjoy-Wolff point. Then for all $w\in\CC_0$, 
 $$\psi^n(w)=an+b\log n+B_n(w)$$
 where 
 \begin{align*}
  a&=\varphi''(1)\\
  b&=(\varphi''(1)^2-2\varphi^{(3)}(1)/3)/\varphi''(1)
 \end{align*}
and $B_n$ is bounded (independently of $n$) on each compact subset of $\CC_0$.
\end{lemma}
\begin{proof}
 We shall use the parabolic model of \cite{BoSh97}: there exists $\sigma:\CC_0\to\CC$ holomorphic such that
 $\sigma\circ\psi=\sigma+a$ and $\sigma$ has the following expansion:
 \begin{equation}\label{eq:orbitsparabolic}
   \sigma(w)=w-b\log(w+1)+B(w)
 \end{equation}
 where $B$ is bounded on $\Omega$, with $\Omega=\CC_0$ if $\Re e(a)>0$ and $\Omega=\CC_0\cap\{\Im m(w)>0\}$
 if $\Re e(a)=0$ and $\Im m(a)>0$ (what we may always assume). Now, let $K$ be a compact subset 
 of $\CC_0$. It can be extracted from the proofs of \cite{BoSh97} that there exists $n_0\in\NN$
 such that $\psi^n(K)\subset \Omega$ for all $n\geq n_0$ and that $1\lesssim|\psi^n(w)|\lesssim n$ uniformly in $n$ and $w\in K$. Therefore, 
 $$\sigma(w)+an=\sigma\circ\psi^n(w)=\psi^n(w)-b\log(\psi^n(w))+C_n(w)$$
 where $|C_n(w)|\lesssim 1$, uniformly in $n$ and $w\in K$. This yields
 \begin{align*}
  \psi^n(w)&=an+b\log\big(\psi^n(w)\big)-C_n(w)+\sigma(w)\\
&=an+b\log(n)+b\log\left(a+\frac{b\log(\psi^n(w))-C_n(w)+\sigma(w)}n\right)-C_n(w)+\sigma(w)\\
  &=:an+b\log n+B_n(w)
 \end{align*}
 where $|B_n(w)|\lesssim 1$ uniformly in $n$ and $w\in K$.
\end{proof}


\begin{theorem}
 Let $\varphi_1,\varphi_2$ be univalent regular parabolic self-maps of $\DD$ with $+1$ as Denjoy-Wolff point. Assume that 
$(\varphi_1''(\alpha),\varphi_1^{(3)}(\alpha))\neq (\varphi_2''(\alpha),\varphi_2^{(3)}(\alpha))$. Then $C_{\varphi_1}$ 
and $C_{\varphi_2}$ are disjointly frequently hypercyclic on $H(\DD)$. 
\end{theorem}
\begin{proof}
 In view of Lemma \ref{lem:orbitsparabolic}, we know that the half-plane models $\psi_i$ satisfy
 $$\psi_i^n(w)=a_i n+b_i\log n+B_{n,i}(w)$$
 with $(a_1,b_1)\neq (a_2,b_2)$ and $B_{n,i}$ uniformly bounded on each compact subset of $\CC_0$. 
 Let $(K_k)$ be a increasing sequence of compact discs of $\CC_0$ such that $\CC_0=\bigcup_k K_k$ and let $(N_k)$
 be a sequence of positive integers verifying 
 $$N_k\geq \sup_{i=1,2}\sup_n\max_{w\in K_k} |B_{n,i}(w)|.$$
 Applying Corollary \ref{cor:setscoholpar}, we get sets $(A_k)\subset\Ald$ satisfying \eqref{eq:corsetscoholpar1}.
 These conditions clearly imply that the sets $\psi_i^n(K_k)$, $i=1,2$, $k\in\NN$, $n\in A_k$, are pairwise disjoint,
 which in turn leads to the desired statement.
\end{proof}

\begin{theorem}\label{thm:coholparsamesame}
 Let $\varphi_1,\varphi_2$ be univalent regular parabolic self-maps of $\DD$ with $+1$ as Denjoy-Wolff point. Assume that 
$(\varphi_1''(\alpha),\varphi_1^{(3)}(\alpha))= (\varphi_2''(\alpha),\varphi_2^{(3)}(\alpha))$. Then $C_{\varphi_1}$ 
and $C_{\varphi_2}$ are not disjointly hypercyclic on $H(\DD)$. 
\end{theorem}
\begin{proof}
 Again, for $i=1,2$, we write
 $$\psi_i^n(w)=an+b\log n+B_{n,i}(w)$$
 with $B_{n,i}$ uniformly bounded on $\Omega$, where $\Omega$ is either the right half-plane or the upper right quarter, and let $\sigma_i$ satisfying \eqref{eq:orbitsparabolic}. 
 Let $C\geq \sup(|B_{n,1}(w)-B_{n,2}(w)|:\ n\in\NN,\ w\in\Omega)$, let $x>32C$ and consider $w_0=x+ix$.
 Observe that Cauchy's formula and \eqref{eq:orbitsparabolic} imply that, for all $\veps>0$, there exists $R>0$ such that
 $|\sigma'(w)-1|<\veps$ provided $\Re e(z)>R$ if $\Omega=\CC_0$ or provided $\Re e(w)>R$ and $\Im m(w)>R$ if
 $\Omega=\{w\in\CC_0:\ \Im m(w)>0\}$. Since the domain $\{w:\ \Re e(w)>R\textrm{ and }\Im m(w)>R\}$ is $\psi_i$-stable provided $R$ is large enough, we may and shall assume that $x$ is so large that, for all $n\in\NN$, 
 $|\sigma_i'(w_0)|/|\sigma_i'(\psi^n(w_0))|\geq 1/2$. Now the parabolic model tells us that $\sigma_i\circ\psi_i^n=\sigma_i+an$.
 Differentiating this equality at $w_0$ entails $|(\psi_i^n)'(w_0)|\geq 1/2$. 
 By Koebe 1/4-theorem, 
 $$\psi_i^n(D(w_0,x/2))\supset D(\psi_i^n(w_0),|(\psi_i^n)'(w_0)|x/8)\supset D(\psi_i^n(w_0),2C).$$
 Since $|\psi_1^n(w_0)-\psi_2^n(w_0)|\leq C$, the sets $\psi_1^n(D(w_0,x/2))$ and $\psi_2^n(D(w_0,x/2))$ are never disjoint,
 which prevents $C_{\psi_1}$ and $C_{\psi_2}$ to be disjointly hypercyclic.  Indeed it is impossible to find $f\in\mathcal H(\CC_0)$ 
 such that $|f\circ\psi_1^n-1|<1/2$ and $|f\circ\psi_2^n-2|<1/2$
 on $D(w_0,x/2)$ for some $n\geq 1$ (see also \cite[Theorem 2.1]{BeMa12})
\end{proof}

\subsection{Two hyperbolic maps with the same boundary fixed point}
We assume in this subsection that $\varphi_1$ and $\varphi_2$ are two univalent hyperbolic self-maps of $\DD$ with the same attractive fixed point $\alpha\in\TT$. 
It is known (see \cite{BeMaPe11}) that, provided $\varphi_1$ and $\varphi_2$ are linear fractional maps of $\DD$,  $C_{\varphi_1}$ and $C_{\varphi_2}$ are disjointly hypercyclic on $H(\mathbb D)$ if and only if $\varphi_1'(\alpha)\neq\varphi_2'(\alpha)$. 

We show that the condition $\varphi_1'(\alpha)\neq\varphi_2'(\alpha)$ is sufficient to ensure disjoint frequent hypercyclicity, without assuming any extra regularity condition
on $\varphi_1$ or $\varphi_2$.

\begin{theorem}\label{thm:hyperbolicsameattractive}
Let $\varphi_1$, $\varphi_2$ be univalent self-maps of $\DD$ with the same Denjoy-Wolff point $\alpha\in\TT$. 
Assume that $\varphi_1$ and $\varphi_2$ are both hyperbolic and that $\varphi_1'(\alpha)\neq\varphi_2'(\alpha)$. Then $C_{\varphi_1}$ and $C_{\varphi_2}$ are disjointly frequently hypercyclic.
\end{theorem}
As before, the proof will rely on Lemma \ref{lem:criterioncomposition} and a separation lemma for the iterates of $\varphi_1$ and $\varphi_2$.
\begin{lemma}\label{lem:separedhyperbolic}
Let $K$ be a compact subset of $\DD$ and let $\eta>0$. Let $\varphi$ be a univalent and hyperbolic self-map of $\DD$ with Denjoy-Wolff point $\alpha\in\TT$. There exists $N\in\NN$ such that, for all $n\geq N$, 
$$\varphi^n(K)\subset\left\{z\in\DD:\ \big((1-\eta)\varphi'(\alpha)\big)^n\leq 1-|z|\leq \big((1+\eta)\varphi'(\alpha)\big)^n\right\}.$$
\end{lemma}
\begin{proof}
We may assume that $K=\overline{D}(0,R)$. By the Schwarz-Pick lemma, $\varphi^n(K)$ is contained in the pseudo-hyperbolic disc 
$$\left\{z\in \DD:\ \frac{|z-\varphi^n(0)|}{|1-\overline{\varphi^n(0)}z|}<R\right\}.$$
This disc is nothing else than the euclidean disc with center $w_n=\frac{(1-R^2)\varphi^n(0)}{1-R^2|\varphi^n(0)|^2}$ and radius $\rho_n=\frac{R(1-|\varphi^n(0)|^2)}{1-R^2 |\varphi^n(0)|^2}$.
Let us write $|\varphi^n(0)|=1-\veps_n$. Then the previous formulae give
$$|w_n|=\frac{(1-R^2)(1-\veps_n)}{1-R^2(1-2\veps_n+o(\veps_n))}=1-\frac{1+R^2}{1-R^2}\veps_n+o(\veps_n)$$
and 
$$\rho_n=\frac{R(1-(1-\veps_n+o(\veps_n)))}{1-R^2(1-2\veps_n+o(\veps_n))}=\frac{R}{1-R^2}\veps_n+o(\veps_n).$$
Therefore we find that there exist $C_1(R)$ and $C_2(R)>0$ such that $\varphi^n(K)$ is contained in the corona 
$$\left\{z\in \DD:\ C_1(R)\veps_n\leq 1-|z|\leq C_2(R)\veps_n \right\}.$$

By the Julia-Caratheodory theorem, since $(\varphi^n(0))$ converges nontangentially to zero, we know that
$$\frac{1-|\varphi^{n+1}(0)|}{1-|\varphi^n(0)|}\xrightarrow{n\to+\infty}\varphi'(\alpha),\ \textrm{namely }\frac{\veps_{n+1}}{\veps_n}\xrightarrow{n\to+\infty}\varphi'(\alpha).$$
This yields
$$\left(\left(1-\frac\eta2\right)\varphi'(\alpha)\right)^n\lesssim \veps_n\lesssim \left(\left(1+\frac\eta2\right)\varphi'(\alpha)\right)^n$$
which in turn establishes the lemma.
\end{proof}

\begin{proof}[Proof of Theorem \ref{thm:hyperbolicsameattractive}]
We fix a sequence $(R_k)$ of positive real numbers tending to 1 and we set $D_k=\overline{D}(0,R_k)$. We may assume that $\varphi_2'(\alpha)<\varphi_1'(\alpha)$ and let $r>1$ be such that $\varphi_2'(\alpha)=\big(\varphi_1'(\alpha)\big)^r$. Let $\delta>0$, $s>1$, $\omega>1$ be such that $(1+\delta)s<(1-\delta)r$ and $r(1+\delta)<(1-\delta)\omega$. We fix a sequence $(N_k)$ of positive integers such that 
\begin{itemize}
\item for all $n\geq N_k$, for $i=1,2$, 
\begin{equation}\label{eq:hyperbolicsameattractive1}
\varphi_i^n(D_k)\subset\left\{z\in\DD:\ \big(\varphi_i'(\alpha)\big)^{(1+\delta)n}\leq 1-|z|\leq \big(\varphi_i'(\alpha)\big)^{(1-\delta)n}\right\}
\end{equation}
(this follows from the previous lemma with $\eta>0$ such that $(1+\eta)\varphi_i'(\alpha)\leq (\varphi_i'(\alpha))^{1-\delta}$ and $(1-\eta)\varphi_i'(\alpha)\leq (\varphi_i'(\alpha))^{1+\delta}$).
\item for all $n\geq N_k$, for $i=1,2$, 
$$\varphi_i^n(D_k)\cap D_k=\varnothing.$$
\end{itemize}
Let also $B=\NN\cap\bigcup_{p\geq 0}[\omega^p,s\omega^p]$. We then consider a family $(A_k)$ of subsets of $B$ with positive lower density such that, for all $k,l\geq 1$, $\min(A_k)\geq k$ and for every $n\in A_k$, $m\in A_l$ with $n\neq m$, then $|n-m|\geq N_k+N_l$. We may conclude exactly as before if we are able to prove that, for all $k,l\geq 1$, all $n\in A_k$, $m\in A_l$ and $i,j\in\{1,2\}$ with $(n,i)\neq (m,j)$, the sets $\varphi_i^n(D_k)$ and $\varphi_j^m(D_l)$ are disjoint. Assume first that $i=j$ and thus $n\neq m$, for instance $n>m$. Then $\varphi_i^{n}(D_k)\cap\varphi_i^m(D_l)=\varnothing$: indeed $D_{\max(k,l)}\cap \varphi_i^{n-m}(D_{\max(k,l)})=\varnothing$ since $n-m\geq \max(N_k,N_l)$. Assume now that $i=1,\ j=2$ and towards a contradiction that $\varphi_1^n(D_k)\cap\varphi_2^m(D_l)\neq\varnothing$. By \eqref{eq:hyperbolicsameattractive1}, the intervals 
$$[\big(\varphi_1'(\alpha)\big)^{(1+\delta)n},\big(\varphi_1'(\alpha)\big)^{(1-\delta)n}]\textrm{ and }
[\big(\varphi_2'(\alpha)\big)^{(1+\delta)m},\big(\varphi_2'(\alpha)\big)^{(1-\delta)m}]$$
are not disjoint. Taking the logarithm and using that $\varphi_2'(\alpha)=\big(\varphi_1'(\alpha)\big)^r$, we get that the intervals 
$[n(1-\delta),n(1+\delta)]$ and $[mr(1-\delta),mr(1+\delta)]$ have to intersect. Let $p,q$ be such that $n\in[\omega^p,s\omega^p]$ and $m\in [\omega^q,s\omega^q]$. Then the intervals $[(1-\delta)\omega^p,(1+\delta)s\omega^p]$ and $[r(1-\delta)\omega^q,r(1+\delta)\omega^q]$ have to intersect. The conditions imposed on $\omega$, $s$ and $\delta$ prevent this.
\end{proof}

We now turn to the case where $\varphi_1$ and $\varphi_2$ are hyperbolic maps of $\DD$, with the same Denjoy-Wolff point $\alpha\in\TT$, 
and the same angular derivative $\varphi'_1(\alpha)=\varphi_2'(\alpha)$. We first extend the result of \cite{BeMaPe11} beyond linear fractional maps,
assuming that $\varphi_1$ and $\varphi_2$ are regular at their Denjoy-Wolff point.

\begin{theorem}\label{thm:coholhypsamesame}
 Let $\varphi_1,\ \varphi_2$ be univalent hyperbolic regular self-maps of $\DD$ with the same Denjoy-Wolff point $\alpha\in\TT$. Assume that $\varphi_1'(\alpha)=\varphi_2'(\alpha)$.
 Then $C_{\varphi_1}$ and $C_{\varphi_2}$ are not disjointly hypercyclic on $H(\DD)$.
\end{theorem}
\begin{proof}
 We move on the half-plane and write $\psi_i(w)=\lambda w+\Gamma_i(w)$ with $|\Gamma_i(w)|\leq M|w|^{1-\veps}$. Let $\sigma_i$
 be the map coming from the linear fractional model, namely $\sigma_i$ is a self-map of $\CC_0$ such that 
 $$\sigma_i\circ\psi_i=\lambda \sigma_i.$$
 The regularity assumption on $\psi_i$ ensures that $\sigma_i$ has finite angular derivative at infinity
 (see \cite[Theorem 2.7]{BrPo03}). Let $x>0$ be large. Since each $\psi_i$ is univalent, we know by Koebe 1/4-theorem that 
 $$\psi_i^n(D(x,x/2))\supset D(\psi_i^n(x), |(\psi_i^n)'(x)|x/8).$$
 Now, from $\sigma_i\circ\psi_i^n =\lambda^n \sigma_i$, we deduce that
 $$(\psi_i^n)'(x)\sigma_i'(\psi_i^n(x))=\lambda^n \sigma_i'(x).$$
 Since the sequence $(\psi_i^n(x))_n$ belongs to some fixed Stolz angle (independent of $x$ and $n$), and since $\sigma_i$ has finite angular derivative at infinity,
 we get 
 $$|(\psi_i^n)'(x)|\geq C_1 \lambda^n$$
 for some $C_1>0$ independent of $x$. Moreover, a look at the proof of Theorem 4.9 of \cite{BoSh97} (see p.60) shows that,
 for all $w\in\CC_0$, $\psi_i^n(w)=\lambda^n w+\lambda^n B_{i,n}(w)$ where $|B_{i,n}(w)|\leq C_2 |w|^{1-\veps}$
 for some $C_2>0$. Therefore, 
 $$|\psi_1^n(x)-\psi_2^n(x)|\leq 2C_2 \lambda^n |x|^{1-\veps}.$$
 This shows that, provided $x$ is large enough, $\psi_1^n(D(x,x/2))$ and $\psi_2^n(D(x,x/2))$ cannot be disjoint. This prevents $C_{\psi_1}$ and $C_{\psi_2}$
 to be disjointly hypercyclic.
\end{proof}

\begin{remark}
Instead of using \cite{BrPo03} and Koebe's theorem, we could apply Lemma \ref{lem:koebehyperbolic} below.
\end{remark}

We now show that we cannot dispense with some regularity condition on the symbols to ensure that we do not have a disjoint (frequent) hypercyclic vector.
We work on the half-plane and we work on an example coming from \cite{Val31}.
\begin{example}
 Let $\psi_1(w)=2w$ and $\psi_2(w)=2w\left(1+\frac1{\log(w+3)}\right)$. Then $C_{\psi_1}$ and $C_{\psi_2}$ are disjointly frequently hypercyclic on $H(\CC_0)$.
\end{example}
\begin{proof}
 We first observe that $\psi_1$ and $\psi_2$ are univalent self-maps of $\CC_0$ and that the corresponding self-maps of the disc $\varphi_1$ and $\varphi_2$ are hyperbolic with $1$
 as Denjoy-Wolff point and $\varphi_1'(1)=\varphi_2'(1)=1/2$. We also know that $\psi_1^n(1)=2^n$ and $\psi_2^n(1)=q_n2^n$
 with $(q_n)$ increasing and $q_n\to+\infty$ (see \cite{Val31}). Let $(R_k)$ be a sequence tending to $+\infty$ and $K_k=\overline{D_\rho}(1,R_k)$
 (closed disc for the hyperbolic distance). Then 
 $$\psi_1^n(K_k)\subset \overline{D_\rho}(\psi_1^n(1),R_k)=2^n \overline{D_\rho}(1,R_k)\subset \{w:\ 2^n\delta_k\leq \Re e(w)\leq 2^n M_k \}$$
 $$\psi_2^n(K_k)\subset  \overline{D_\rho}(\psi_2^n(1),R_k)=\psi_2^n(1) \overline{D_\rho}(1,R_k)\subset \{w:\ q_n 2^n \delta_k\leq \Re e(w)\leq q_n 2^n M_k \}$$
where $\delta_k$ and $M_k$ only depend on $k$  (see \cite[p.68, (17)]{Sha93}).
 Let $\phi_1(n)=n$ and $\phi_2(n)=n+\lfloor \frac{\log q_n}{\log 2}\rfloor$. Let also, for all $k\geq 1$, $N_k$ be an integer such that 
 $$N_k\geq \frac{\log M_k}{\log 2}-\frac{\log\delta_k}{\log 2}+2.$$
 We apply Corollary \ref{cor:sets} to $\phi_1,\phi_2$ and $(N_k)$ and we observe that \eqref{eq:corsets1} implies that the sets $\psi_i^n(K_k)$, $i\in\{1,2\}$, 
 $k\geq 1$, $n\in A_k$, are pairwise disjoint. 

\end{proof}

\subsection{Hyperbolic and parabolic maps}

We turn to the last case: when one map is hyperbolic and the other one is parabolic, with the same Denjoy-Wolff point. 
We are reduced to prove the following result.

\begin{theorem}\label{thm:coholparhyp}
Let $\varphi_1,\varphi_2$ be regular univalent self-maps of $\DD$ with $+1$ as Denjoy-Wolff point. Assume that $\varphi_1$ is hyperbolic and that $\varphi_2$ is parabolic. Then $C_{\varphi_1}$ and $C_{\varphi_2}$ are disjointly frequently hypercyclic if and only if $\varphi_2''(1)\in i\mathbb R$.
\end{theorem}

Half of the proof has already been done: if $\varphi_2''(1)\in i\mathbb R$, then $(\varphi_2^n(0))$ converges tangentially to $\alpha$ whereas $(\varphi_1^n(0))$ converges
nontangentially to $\alpha$. Hence Corollary \ref{cor:tangentialandnontangential} does the job. The converse implication uses some ideas of the proofs of Theorem \ref{thm:coholparsamesame} and of Theorem \ref{thm:coholhypsamesame}. Nevertheless, using the work done throughout this section, it is not difficult to see that, given any compact subset $K$ of $\DD$,  $\varphi_1^n(K)$ and $\varphi_2^n(K)$ are pairwise disjoint provided $n$ is large enough. Therefore, by \cite[Theorem 2.1]{BeMa12}, $C_{\varphi_1}$ and $C_{\varphi_2}$ are disjointly hypercyclic. To contradict the disjoint {\it frequent} hypercyclicity, we will need a refinement of Koebe 1/4-theorem specific to our context. For $w\in\CC_0$, $w=xe^{i\theta}$ with $x>0$ and $\theta\in(-\pi/2,\pi/2)$, for $\delta\in(0,1)$ and $c>0$, we denote by $\mathcal R(w,\delta,c)$ the rectangle whose apexes are $w\pm(1-\delta)x\cos\theta\pm icx$. 

\begin{lemma}\label{lem:koebehyperbolic}
Let $\theta\in (-\pi/2,\pi/2)$, let $\varphi$ be a hyperbolic univalent regular self-map of $\DD$ with $+1$ as Denjoy-Wolff point and let $\lambda=1/\varphi'(1)$. For all $\delta\in(0,1/4)$, there exists  $x_0\in(0,+\infty)$ such that, for all $x\geq x_0$, for all $n\geq 1$,
setting $w_0=xe^{i\theta}$,
$$\psi^n\big(\mathcal R(w_0,\delta,4)\big)\supset \mathcal R(\lambda^n w_0,4\delta,1).$$
\end{lemma}
\begin{proof}
Let $w\in\partial R(w_0,\delta,4)$. We are going to prove that either
$$|\Re e(\psi^n(w)-\psi^n(w_0))|\geq (1-2\delta)\lambda^n x\cos\theta\textrm{ or }|\Im m(\psi^n(w)-\psi^n(w_0))|\geq 2\lambda^n x$$
provided $x$ is large enough. Indeed, recall that, for all $w\in\CC_0$, $\psi^n(w)=\lambda^n w+\lambda^n B_n(w)$ with $|B_n(w)|\leq C|w|^{1-\veps}$ for some $C>0$ independent of $w$ and $n$. Assume first that $w=w_0\pm(1-\delta)x\cos\theta+icx$ with $c\in[-4,4]$. Then we get
\begin{align*}
|\Re e(\psi^n(w)-\psi^n(w_0))|&\geq (1-\delta)\lambda^n x\cos\theta-C_1\lambda^n x^{1-\veps}\\
&\geq (1-2\delta)\lambda^n x\cos \theta
\end{align*}
provided $x$ is large enough. On the other hand, if $w=w_0+dx\cos\theta\pm 4ix$, with
$d\in [-(1-\delta),(1-\delta)]$, then
\begin{align*}
|\Im m(\psi^n(w)-\psi^n(w_0))|&\geq 4\lambda^n x-C_2\lambda^n x^{1-\veps}\\
&\geq 2\lambda^n x
\end{align*}
if $x$ is large enough. It then follows that $\psi^n(\mathcal R(w_0,\delta,4))$ contains the rectangle with apexes $\psi^n(w_0)\pm\lambda^n (1-2\delta)x\cos\theta\pm 2\lambda^n x$. Now, $|\psi^n(w_0)-\lambda^n w_0|\leq C\lambda^n x^{1-\veps}$ so that, again provided $x$ is large enough, this rectangle contains $\mathcal R(\lambda^n w_0,4\delta,1)$.
\end{proof}

\begin{proof}[Proof of Theorem \ref{thm:coholparhyp}]
We assume $\varphi_2''(1)\notin i\mathbb R$ and we write it $\varphi_2''(1)=\rho e^{i\theta}$. We argue by contradiction and prove
a slightly stronger result: any $f\in H(\mathbb C_0)$ which is frequently hypercyclic for $C_{\psi_2}$ cannot be hypercyclic
for $C_{\psi_1}$. Hence, let $f\in H(\mathbb C_0)$ be a frequently hypercyclic vector for $C_{\psi_2}$, let 
$$A=\{n\in\NN:\ |f\circ\psi_2^n (1)|<1/2\}$$
which has positive lower density and let $\delta\in(0,1/8)$. Let  $w_0=xe^{i\theta}$ be given by Lemma \ref{lem:koebehyperbolic} for $\varphi=\varphi_1$. For $p\geq 1$, we define
$$F_p=\left\{n\leq\frac{\lambda^p x}{\rho}:\ \psi_2^n(1)\in \mathcal R(\lambda^p w_0,4\delta,1)\right\}.$$
Let us show that $I_p:=\left[\frac{8\delta\lambda^px}{\rho},\frac{\lambda^px}{\rho}\right]\cap\NN$ is contained in $F_p$ provided $p$ is large enough. Indeed, let $n\in I_p$ and recall that $\psi_2^n(1)=\rho e^{i\theta}n+b\log n+B_n(1)$ with $|B_n(1)|\leq C$. Hence, 
$$-|b|\log n-C\leq \lambda^p \Re e(w_0)-\Re e(\psi_2^n(1))\leq  (1-8\delta)x\cos\theta\lambda^p+|b|\log n+C$$
which yields
$$\big| \Re e(w_0)-\Re e(\psi_2^n(1))\big|\leq (1-4\delta)x\cos\theta\lambda^p$$
if $p$ is sufficiently large.
Similarly, assuming $\sin(\theta)\geq 0$, we also have
$$-|b|\log n-C\leq \lambda^p \Im m(w_0)-\Im m(\psi_2^n(1))\leq  (1-8\delta)x\sin\theta\lambda^p+|b|\log n+C$$
so that 
$$\big|  \lambda^p \Im m(w_0)-\Im m(\psi_2^n(1)) \big| \leq \lambda^p x.$$
Now, if $f$ is a hypercyclic vector for $C_{\psi_1}$, there exists $p$ as large as we want such that $|f\circ\psi_1^p-1|<1/2$ on $\mathcal R(w_0,\delta,4)$ so that $|f(w)|>1/2$ on $\mathcal R(\lambda^p w_0,4\delta,1)$.  Hence  $A\cap F_p=\varnothing$ which shows
that $\ldens(A)\leq 8\delta$. Since $\delta>0$ is arbitrary, this is a contradiction. 
\end{proof}

\providecommand{\bysame}{\leavevmode\hbox to3em{\hrulefill}\thinspace}
\providecommand{\MR}{\relax\ifhmode\unskip\space\fi MR }
\providecommand{\MRhref}[2]{%
  \href{http://www.ams.org/mathscinet-getitem?mr=#1}{#2}
}
\providecommand{\href}[2]{#2}

\end{document}